\documentclass[a4paper, 10pt, twoside, notitlepage]{amsart}

\usepackage[utf8]{inputenc}
\usepackage{color}
\usepackage{amsmath} 
\usepackage{amssymb} 
\usepackage{amsthm}
\usepackage{geometry}
\usepackage{graphicx}
\usepackage{esint}
\usepackage[ocgcolorlinks,linkcolor=blue]{hyperref}

\theoremstyle{plain}
\newtheorem{thm}{Theorem}[section]
\newtheorem{prop}{Proposition}[section]
\newtheorem{lem}[prop]{Lemma}
\newtheorem{cor}[prop]{Corollary}

\newtheorem{defi}[prop]{Definition}
\newtheorem{rmk}[prop]{Remark}

\numberwithin{equation}{section}
\newcommand {\R} {\mathbb{R}} 
 \newcommand {\N} {\mathbb{N}}
 
\newcommand {\p} {\partial}

\newcommand {\D} {\Delta}

\newcommand {\supp} {\text{supp}}

\DeclareMathOperator{\Cof}{Cof}

\DeclareMathOperator {\dist} {dist}

\DeclareMathOperator{\F} {\mathcal{F}}

\DeclareMathOperator{\codim}{codim}
\DeclareMathOperator{\corank}{corank}
\DeclareMathOperator{\diam}{diam}

\title[The fractional Calder\'on problem with drift]{The Calder\'on problem for the fractional Schr\"odinger equation with drift}

\author[M. Ceki\'c]{Mihajlo Ceki\'c}
\address{Max-Planck Institute for Mathematics, Vivatsgasse 7, 53111, Bonn, Germany}
\curraddr{}
\email{m.cekic@mpim-bonn.mpg.de}

\author[Y.-H. Lin]{Yi-Hsuan Lin}
\address{Department of Mathematics and Statistics, University of Jyv\"askyl\"a, 40014 Jyv\"askyl\"a, Finland}
\curraddr{}
\email{yihsuanlin3@gmail.com}

\author[A. R\"uland]{Angkana R\"uland}
\address{Max-Planck Institute for Mathematics in the Sciences, Inselstraße 22, 04103 Leipzig, Germany}
\curraddr{}
\email{rueland@mis.mpg.de}

\begin{document}
	
	\maketitle
	
	\begin{abstract}
		We investigate the Calder\'on problem for the fractional Schr\"odinger equation with drift, proving that the unknown drift and potential in a bounded domain can be determined simultaneously and uniquely by an infinite number of exterior measurements. In particular, in contrast to its local analogue, this nonlocal problem does \emph{not} enjoy a gauge invariance.
The uniqueness result is complemented by an associated logarithmic stability estimate under suitable apriori assumptions. Also uniqueness under finitely many \emph{generic} measurements is discussed. Here the genericity is obtained through \emph{singularity theory} which might also be interesting in the context of hybrid inverse problems. Combined with the results from \cite{GRSU18}, this yields a finite measurements constructive reconstruction algorithm for the fractional Calder\'on problem with drift.
The inverse problem is formulated as a partial data type nonlocal problem and it is considered in any dimension $n\geq 1$.

		\medskip
		
		
		
	\end{abstract}

	\tableofcontents

	\section{Introduction}
In this article, we consider an inverse problem for a nonlocal Schrödinger equation with drift. Here we seek to study the uniqueness, stability and reconstruction properties in analogy to its local counterpart, which is a type of magnetic Schr\"odinger equation and had been investigated in Nakamura-Sun-Uhlmann \cite{nakamura1995global}. As one of our main results, we prove that: In contrast to its local counterpart, for the fractional Calder\'on problem with drift there is \emph{no} gauge invariance present (see Theorem \ref{thm main}). In particular, this poses an obstruction in possibly extracting information from the nonlocal inverse problem for its local analogue (as $s\to 1$). As our second main result, we prove a \emph{generic}, finite measurements reconstruction, which might also be of interest in the context of hybrid inverse problems (see Theorem \ref{prop:finite_meas} and the following discussions). As a key tool for this, we rely on singularity theory from Whitney \cite{whitney}.

\vskip0.5cm
	
\textbf{The classical Calder\'on problem with drift.} 	
	Before turning to the \emph{nonlocal} problem, let us recall its \emph{local} analogue and the known results on this: In the classical Calder\'on problem for the magnetic Schr\"odinger equation, the objective is to determine the drift and potential coefficients simultaneously. More precisely, for  $n\geq 3$, let $\Omega\subset \R^n$ be a bounded Lipschitz domain, then consider the following Dirichlet boundary value problem 
	\begin{align}
	\begin{split}\label{local drift equation}
	(-\Delta + b\cdot \nabla +c)u & =0 \text{ in }\Omega,\\
	u &=f  \text{ on }\p \Omega,
	\end{split}
	\end{align} 
	where $b$ and $c$ are sufficiently smooth functions which vanish on $\partial \Omega$. 
	Assuming the well-posedness of the boundary value problem \eqref{local drift equation}, one can define boundary measurements given by the \emph{Dirichlet-to-Neumann map} (abbreviated as the DN map in the rest of this paper) 
	$$
	\Lambda_{b,c}:H^{1/2}(\partial \Omega) \to H^{-1/2}(\partial \Omega) \text{ with }\Lambda_{b,c}:f \mapsto \dfrac{\partial u}{\partial \nu},
	$$ 
	where $u\in H^1(\Omega)$ is the unique solution to \eqref{local drift equation} and $\nu$ is the unit outer normal on $\partial \Omega$. 
	
	The Calder\'on problem for \eqref{local drift equation} consists of trying to recover the unknown coefficients $b$, $c$ (which are assumed to be in appropriate function spaces) by the information encoded in the boundary measurement operator $\Lambda_{b,c}$ on $\partial \Omega$. In the case of the local magnetic Schrödinger equation, there is however an \emph{intrinsic obstruction} to the unique identification of these coefficients:
To observe this, consider the substitution $v= e^{\phi} u$ which results in an equation for $v$ which is of a similar form as \eqref{local drift equation}:
	\begin{align}
	\label{eq:gauge}
	\begin{split}
	-\Delta v + (b+ 2\nabla \phi) \cdot \nabla v + (c+\Delta \phi - b\cdot \nabla \phi -|\nabla \phi|^2) v &= 0 \mbox{ in } \Omega,\\
	v&= e^{\phi} f \mbox{ on } \partial \Omega.
	\end{split}
	\end{align}
	If now $\phi = 0$, $\partial_{\nu} \phi = 0$ on $\partial \Omega$, then the DN maps for the old and the new equations \eqref{local drift equation} and \eqref{eq:gauge} coincide; there is a hidden \emph{gauge invariance}. As a consequence, one cannot expect to be able to recover the full information on the coefficients $b$ and $c$ from the knowledge of the Dirichlet-to-Neumann map
 $\Lambda_{b,c}$ on $\p \Omega$. 
 
As a matter of fact, one can at most hope to recover $b$ and $c$ up to the described gauge invariance for the inverse boundary value problem with respect to \eqref{local drift equation}. 
This is indeed the case (c.f. \cite[Theorem 5.4.1]{isakov2017inverse}): If $b_j$ and $c_j$ are compactly supported in a simply connected domain for $j=1,2$, and if their DN maps coincide on the boundary $\partial \Omega$, then one can obtain uniqueness up to the described gauge invariance:
	\begin{align}\label{gauge invariance}
	\mathrm{curl}b_1=\mathrm{curl}b_2 \text{ and }4c_1+b_1\cdot b_1-2\mathrm{div}b_1=4c_2+b_2\cdot b_2-2\mathrm{div}b_2 \text{ in }\Omega.
	\end{align}
We again emphasize that this does not allow us to recover the full fields $b,c$ but only allows one to obtain information up to the above gauge invariance. 

Variants of this inverse boundary value problem have also been studied in \cite{sun1993inverse, salo2004inverse, tolmasky1998exponentially} for the symmetric magnetic Schrödinger operator, in \cite{krupchyk2014uniqueness, haberman2016unique} for the low regularity setting and in \cite{ferreira2007determining} for the setting in which only partial data are available. The setting of flexible geometries was studied in \cite{cekic2017calderon, krupchyk2017inverse}. The case of systems was considered in \cite{eskinsystems01} and Yang-Mills potentials with arbitrary geometry in \cite{cekicym17}. Stability results can be found in \cite{tzou2008stability, joud2009stability} and reconstruction results are given in \cite{salo2006semiclassical}. For more detailed discussions of inverse problems for the magnetic Schr\"odinger equation, we refer to the book \cite[Chapter 5]{isakov2017inverse} as well as to the articles \cite{salo2006inverse,uhlmann2009electrical} and their bibliographies. 

\vskip0.5cm

\textbf{The fractional Calder\'on problem with drift.}
Keeping the situation of the local problem with $s=1$ in the back of our minds, we turn to the Calder\'on problem for the analogous fractional Schr\"odinger equation with drift, which is a \emph{nonlocal} inverse problem. This inverse problem should be regarded as a generalization of the fractional Calder\'on problem, which had first been introduced and investigated in Ghosh-Salo-Uhlmann \cite{ghosh2016calder}. In the sequel, we describe this problem more precisely.

Assume that $\Omega \subset \R^n$ is a bounded Lipschitz domain for $n\geq 1$ and let $\frac{1}{2}<s<1$ (so that the fractional nonlocal operator is dominant). 
	Given a drift $b\in W^{1-s,\infty}(\Omega)^n$ 
	and a potential $c\in L^{\infty}(\Omega)$, we consider the following fractional exterior value problem 
	\begin{align}
	\label{eq:main}
	\begin{split}
	((-\D)^s + b \cdot \nabla + c)u & = 0 \mbox{ in } \Omega,\\
	u& = f \mbox{ in } \Omega_e:=\R^n \setminus \overline{\Omega},
	\end{split}
	\end{align}
	with some suitable exterior datum $f$ (we will present a rigorous mathematical formulation of this in Section \ref{sec:2}, c.f. also Remark \ref{rmk:conditions_s}). Here the fractional Laplacian $(-\Delta)^s$ is given by 
	$$
	(-\Delta)^s u:= \mathcal{F}^{-1}\left\{|\xi|^{2s}\widehat{u}(\xi)\right\}, \text{ for }u\in H^s(\R^n),
	$$
	where $\widehat{u}=\mathcal{F}u$ denotes the Fourier transform of $u$. 
We assume that this problem and its ``adjoint problem'' are well-posed, which is guaranteed by imposing the eigenvalue condition that 

		\begin{align}\label{eigenvalue condition}
	\begin{split}
	& w\in H^s(\R^n)\text{ is a solution of }(-\Delta)^sw+b\cdot \nabla w +cw=0\text{ in }\Omega \mbox{ with } w = 0 \mbox{ in } \Omega_e,\\
	&\text{then we have }w\equiv 0.
	\end{split}
	\end{align}
	
In the sequel, we will always suppose that this condition is satisfied. As in the classical Calder\'on problem with drift, we are interested in recovering the drift coefficient $b \in W^{1-s,\infty}(\Omega)^n$ and the potential $c \in L^{\infty}(\Omega)$ simultaneously from the associated DN map. With slight abuse of notation (for a precise definition we refer to Section \ref{sec:2}), the DN map associated with the nonlocal problem can be thought of as the mapping
\begin{align}\label{DN map}
\Lambda_{b,c}: \widetilde{H}^{s}(\Omega_e) \rightarrow (\widetilde{H}^{s}(\Omega_e))^{\ast}, \ f \mapsto (-\D)^s u|_{\Omega_e},
\end{align} 
where $u$ is the solution to \eqref{eq:main} with exterior data $f$.

The aim of this work is to prove the \emph{global uniqueness} and \emph{stability} of the drift coefficient and the potential for this nonlocal inverse problem. This is in strong contrast to the local case, i.e., the case $s=1$, as in the nonlocal setting the gauge invariance \eqref{gauge invariance} which presented an obstruction to global uniqueness in the local case disappears.

\subsection{The main results}
\label{sec:main_results}	
Let us formulate our main results. As a first property we obtain the global uniqueness of the drift coefficient and the potential for the nonlocal fractional Calder\'on problem with drift:
	
\begin{thm}[Global uniqueness]\label{thm main}
For $n\geq 1$, let $\Omega \subset \R^n$ be a bounded open Lipschitz domain, and let $\frac{1}{2}<s<1$. Let $b_j\in W^{1-s,\infty}(\Omega)^n$ be two drift fields, and $c_j\in L^\infty (\Omega)$ be potentials for $j=1,2$. Given arbitrary open sets $W_1,W_2 \subset \Omega_e$, suppose that the DN maps for the equations 
		$$
		((-\Delta)^s+b_j\cdot \nabla +c_j)u_j=0 \text{ in }\Omega
		$$
		satisfy
		$$
		\Lambda_{b_1,c_1}f|_{W_2}=\Lambda_{b_2,c_2}f|_{W_2}, \text{ for any }f\in C^\infty _c(W_1).
		$$
Then $b_1=b_2$ and $c_1=c_2$ in $\Omega$.
	\end{thm}
	
	Theorem \ref{thm main} can be regarded as a \emph{partial data} result for our nonlocal inverse problem. Different from the local case, i.e., $s=1$, there is \emph{no gauge invariance} and thus \emph{no intrinsic obstruction} to uniqueness in this nonlocal Calder\'on problem. We expect that it is possible to improve the regularity assumptions on the drift coefficient and the potential. As our main focus in the present article is however on the striking differences between the local and the nonlocal problems in terms of the existence/absence of a gauge, we do not elaborate on this here but postpone this to a future work.
	
As in previous results on the fractional Calder\'on problem (c.f. \cite{ghosh2016calder}), in this work, we are not using \emph{complex geometrical optics} solutions. Instead, we rely on the following approximation property.

	\begin{thm}[Runge approximation]
		\label{thm: Runge approximation}
		For $n\geq 1$ and $\frac{1}{2}<s<1$, let $\Omega \subset \R^n$ be a bounded open Lipschitz set and $\Omega_1\subset \R^n$ be an arbitrary open set containing $\Omega$ with $\mathrm{int}(\Omega_1 \setminus \overline{\Omega})\neq \emptyset$. 
		\begin{itemize}
			\item[(a)] Let $b\in W^{1-s,\infty}(\Omega)^n$ 
and $c\in L^\infty(\Omega)$. Then, for any $g\in L^2(\Omega)$ and $\epsilon>0$, one can find a solution $u_\epsilon \in H^s(\R^n)$ of 
			\begin{align*}
			\left((-\Delta)^s+b\cdot \nabla +c\right)u_\epsilon =0 \text{ in }\Omega, \text{ with }\mathrm{supp}(u_\epsilon)\subset \overline{\Omega_1}
			\end{align*}
			such that
			\begin{align*}
			\|u_\epsilon - g \|_{L^2(\Omega)}<\epsilon.
			\end{align*}
			
			\item[(b)] If we further assume that $\Omega$ has a $C^\infty$-smooth boundary, $b\in C^\infty_c(\Omega)^n$ and $c\in C^\infty_c(\Omega)$ with $\supp(b), \supp(c)\Subset \Omega$, given any $g\in C^\infty(\overline{\Omega})$, $\epsilon>0$ and $k\in \N$, then there exists a solution  $u_{\epsilon}\in H^{s}(\R^n)$ of 
			\begin{align*}
			\left((-\Delta)^s+b\cdot \nabla +c\right)u_\epsilon =0 \text{ in }\Omega, \text{ with }\mathrm{supp}(u_\epsilon)\subset \overline{\Omega_1}
			\end{align*}
			such that
			\begin{align*}
			\|d^{-s}(x) u_{\epsilon} -g\|_{C^k(\Omega)} \leq \epsilon.
			\end{align*}
			The function $d(x)$ is any $C^\infty$-smooth function defined in $\overline{\Omega}$ such that $d>0$ in $\Omega$ and $d(x)=\mathrm{dist}(x,\partial \Omega)$ whenever $x$ is near $\partial \Omega$.
		\end{itemize}
		
	\end{thm}
	
A more quantitative Runge approximation for the fractional Schrödinger equation with drift will be discussed in Section \ref{sec:stability} in the context of stability estimates.	
	
	\begin{rmk}
		The qualitative Runge approximation property as a key tool for studying fractional Schrödinger type inverse problems had been introduced in \cite{ghosh2016calder}. In order to infer such a result, the authors of \cite{ghosh2016calder} built on the unique continuation property for fractional Schrödinger equations in the form of Carleman estimates which had been derived in \cite{ruland2015unique}, c.f. also \cite{FF14,FF15,GRSU18,R17,Seo,yu2017unique} for related unique continuation results for fractional Schrödinger equations. For variable coefficient fractional Schrödinger operators, the authors of \cite{ghosh2017calderon} utilized Almgren's frequency function to derive such a property. In the context of nonlocal elliptic equations, this had earlier been employed by \cite{FF14,yu2017unique}, c.f. also Section 7 in \cite{ruland2015unique} for the derivation of unique continuation properties with variable coefficients.
	\end{rmk}

Let us put these results into the context of the literature on the fractional Calder\'on problem: The problem was first introduced by \cite{ghosh2016calder}, where the authors treated the case with $c\in L^{\infty}(\Omega)$, $b=0$ and $s\in(0,1)$, and proved a global uniqueness result for $c$. For more general nonlocal variable coefficient Schrödinger operators, the fractional Calder\'on problem was studied in \cite{ghosh2017calderon}. The techniques based on Runge approximation are strong enough to deal with the case of semilinear equations \cite{LL18} and low regularity, almost critical function spaces for the potential \cite{RS17}. Even \emph{single} measurement results are possible \cite{GRSU18} (c.f. the discussion below). Moreover, these techniques have been extended to other nonlocal problems \cite{Ru17quantitative, cao2018determining} in a slightly different context. Also, for positive potentials, monotonicity inversion formulas have been successfully discovered in \cite{harrach2017nonlocal-monotonicity}. Very recently and independently from our work, uniqueness results have been obtained for equations with nonlocal lower order contributions \cite{BGU18}.
	
In addition to uniqueness, stability is of central importance in inverse problems. Stability results for the fractional Calder\'on problem were first obtained in \cite{RS17, RS18exponential}, where optimal logarithmic stability estimates had been derived (c.f. also \cite{ruland2018lipschitz} for improvements of this if structural apriori conditions like the finiteness of the underlying function space are satisfied).
It is possible to extend the logarithmic estimates to the setting of the fractional Calder\'on problem with drift. Here we obtain the following result:
	
	\begin{thm}[Logarithmic stability]
		\label{thm:stab}
		Let $s\in (\frac{1}{2},1)$, $\Omega \subset \R^n$, $n\geq 1$ be a bounded open smooth domain. Let $\overline{W}_1, \overline{W}_2 \subset \Omega_e$ be such that $\overline{\Omega} \cap \overline{W}_1 = \overline{\Omega} \cap \overline{W}_2 = \emptyset$. 
		Assume that for some constants $M>0$, $\delta>0$
		
		\begin{align*}
		\|b_j\|_{W^{1-s+\delta, \infty}(\Omega)}
		+ \|c_j\|_{H^s(\Omega)} + \|c_j\|_{W^{1,n+\delta}(\Omega)} \leq M,
		\end{align*}
		and that $\supp(b_j), \supp(c_j) \Subset \Omega$ for $j=1,2$. 
		Then for some constants $\mu>0$ and $C>0$ which depend on $\Omega, W, n, s, M, \delta$, we have
		\begin{align*}
		\|c_1-c_2\|_{H^{-s}(\Omega)} + \|b_1-b_2\|_{H^{-s}(\Omega)}
		\leq C\left|\log(\|\Lambda_{b_1,c_1}-\Lambda_{b_2,c_2}\|_{\ast})\right|^{-\mu},
		\end{align*} 
		if $\|\Lambda_{b_1,c_1}-\Lambda_{b_2,c_2}\|_{\ast} \leq 1$,
		where $\|A\|_{\ast}:= \sup\{ (A f_1, f_2)_{W_2}: \ f_1 \in \widetilde{H}^s(W_1), \ f_2 \in \widetilde{H}^s(W_2)  \}$.
	\end{thm}  
	
	As in \cite{RS17} this relies on quantitative Runge approximation arguments, which are derived from quantitative unique continuation properties. We adapt the arguments from \cite{RS17} to infer these results for the fractional Schr\"odinger equation with drift.
	
Last but not least, based on the higher order Runge approximation property (Theorem \ref{thm: Runge approximation} (b)), we can deduce \emph{finite measurements} uniqueness results for the fractional Calder\'on problem with drift. 
	
\begin{thm}[Finite measurements uniqueness]
		\label{prop:finite_meas}
		Let $\Omega \subset \R^n$ be a bounded domain with a $C^\infty$-smooth boundary. Let $W \subset \Omega_e$ be an open, smooth set such that $\overline{W}\cap \overline{\Omega} = \emptyset$. Let $s\in (\frac{1}{2},1)$ and assume that $b_j\in C^{\infty}_c(\Omega)^n$, $c_j\in C^{\infty}_c(\Omega)$ satisfy \eqref{eigenvalue condition} with $\supp(b_j), \supp(c_j)\Subset \Omega$ for $j=1,2$.  
There exist $n+1$ exterior data $f_1,\dots,f_{n+1}\in C^{\infty}_c(W)$ such that if
		\begin{align*}
		\Lambda_{b_1,c_1}(f_l) = \Lambda_{b_2,c_2}(f_l) \mbox{ for } l \in\{1,\dots,n+1\},
		\end{align*}
		then $b_1 = b_2$ and $c_1=c_2$.
		Moreover, the set of exterior data $f_1,\dots,f_{n+1}$, which satisfies this property forms an open and dense subset in $C^{\infty}_{ c}(W)$.
	\end{thm}  
		
		This is analogous to the single measurement results in \cite{GRSU18} for the fractional Schr\"odinger equation, c.f. also \cite{cao2017simultaneously} for a single measurement result on the detection of an embedded obstacle. However, compared to \cite{GRSU18} a word of caution is needed here: In contrast to the result from \cite{GRSU18} it is \emph{not} possible to work with an arbitrary nontrivial set of measurements $f_1,\dots,f_{n+1}$. The data $f_1,\dots,f_{n+1}$ have to be chosen appropriately from a set which depends on the unknowns $b$, $c$. This is similar to results on hybrid inverse problems, c.f. \cite{BU13, A15}. 
		
		While the dependence of the admissible exterior data on the unknown drift field and potential seems like a serious restriction at first sight, we emphasise that by proving that the data $f_1,\dots,f_{n+1}$ can be chosen in an \emph{open} and \emph{dense} set in $C_{c}^{\infty}(W)$, we show that the set of admissible exterior data is very large: Given a (random) exterior measurement $f_1,\dots,f_{n+1} \in C_c^{\infty}(W)$, our result states that an arbitrarily small perturbation of this yields an admissible exterior datum from which we can reconstruct the drift field $b$ and potential $c$. This might also be of interest in the setting of hybrid inverse problems for which we could not find a statement on an \emph{open} \emph{and} \emph{dense} set of admissible measurements. We plan to address this in future research.

For an overview about the fractional Calder\'on problem, we refer to the survey  \cite{salo2017fractional}.

\subsection{Outline of the remaining article}	
The paper is organized as follows. In Section \ref{sec:2}, we review the notion of a weak solutions of the fractional Schr\"odinger equation with drift. With this at hand, we define the DN map rigorously. Section \ref{sec:3} demonstrates the $L^2$-Runge approximation property, which proves Theorem \ref{thm: Runge approximation}(a). We will prove the global uniqueness result of Theorem \ref{thm main} in Section \ref{sec:4}, which shows that the nonlocal Calder\'on problem does not enjoy a gauge invariance in contrast to its local analogue. In Section \ref{sec:stability}, we also prove the stability result of Theorem \ref{thm:stab} for the fractional Calder\'on problem with drift and potential with respect to the associated DN maps. In Section 6, in the end of this note, we present the proof of Theorem \ref{prop:finite_meas}, where we prove several points on the reconstruction from finitely many exterior measurements. In addition, we study generic unique determination results via singularity theory in the Appendix, which is useful to understand the open and dense subset for the exterior data stated in Theorem \ref{prop:finite_meas}.

	\section{The fractional Schr\"odinger equation with drift}\label{sec:2}
	In this section, we recall the relevant function spaces, prove the well-posedness of the fractional Schr\"odinger equation with drift and introduce and derive properties of the Dirichlet-to-Neumann map associated with \eqref{eq:main}.

	\subsection{Preliminaries}
		We begin by recalling the relevant fractional Sobolev spaces on (bounded) domains. We define the $L^2$-based fractional Sobolev spaces as follows: for $0<s<1$, we consider the fractional Sobolev spaces $H^s(\R^n)=W^{s,2}(\R^n)$ with the norm 
	\begin{align*}
	\|u\|_{H^s(\R^n)}:=\left\|\mathcal{F}^{-1}\left\{\langle \xi \rangle^s \widehat{u} \right\}\right\|_{L^2(\R^n)},
	\end{align*}
	where $\left\langle \xi\right\rangle =(1+|\xi|^{2})^{\frac{1}{2}}$.
	Let $\mathcal{O} \subset \mathbb{R}^{n}$ be an arbitrary open set and $0<s<1$, then we define:
	\begin{align*}
	H^{s}(\mathcal{O}) & :=\{u|_{\mathcal{O}};\,u\in H^{s}(\mathbb{R}^{n})\},\\
	\widetilde{H}^{s}(\mathcal{O}) & :=\text{closure of \ensuremath{C_{c}^{\infty}(\mathcal{O})} in \ensuremath{H^{s}(\mathbb{R}^{n})}},\\
	H_{0}^{s}(\mathcal{O}) & :=\text{closure of \ensuremath{C_{c}^{\infty}(\mathcal{O})} in \ensuremath{H^{s}(\mathcal{O})}},
	\end{align*}
	and 
	\[
	H_{\overline{\mathcal{O}}}^{s}:=\{u\in H^{s}(\mathbb{R}^{n}); \quad \text{ with }\mathrm{supp}(u)\subset\overline{\mathcal O}\}.
	\]
	The norm of $H^{s}(\mathcal{O})$ is denoted by 
	\[
	\|u\|_{H^{s}(\mathcal{O})}:=\inf\left\{ \|v\|_{H^{s}(\mathbb{R}^{n})};v\in H^{s}(\mathbb{R}^{n})\mbox{ and }v|_{\mathcal{O}}=u\right\} .
	\]
	It is known that $\widetilde{H}^{s}(\mathcal{O})\subseteq H_{0}^{s}(\mathcal{O})$,
	and that $H_{\overline{\mathcal{O}}}^{s}$ is a closed subspace of
	$H^{s}(\mathbb{R}^{n})$. Further we have for arbitrary open sets $\mathcal{O}$
	\begin{align*}
\left(H^{s}(\mathcal{O})\right)^{\ast}=\widetilde{H}^{-s}(\mathcal{O})\mbox{ and }\left(\widetilde{H}^{s}(\mathcal{O})\right)^{\ast}=H^{-s}(\mathcal{O}).
	\end{align*}
	
	\begin{rmk}\label{rmk on Sobolev}
		When $\mathcal{O}\subset \R^n$ is a bounded Lipschitz domain, we have that for any $s\in\mathbb{R}$, 
		\[
		\begin{split} & \widetilde{H}^{s}(\mathcal{O})=H_{\overline{\mathcal{O}}}^{s}\subseteq H_{0}^{s}(\mathcal{O}).
		\end{split}
		\] 
If $s>-\frac{1}{2}$ and $s\notin \{\frac{1}{2},\frac{3}{2},\dots\}$ the last inclusion also becomes in equality.
	\end{rmk}

We further denote the homogeneous fractional Sobolev spaces as $\dot{H}^s(\R^n)$, where
\begin{align*}
\dot{H}^s(\R^n):=\{u: \R^n \rightarrow \R; \  \|\F^{-1} \{ |\xi|^s \hat{u} \} \|_{L^2(\R^n)} <\infty\}.
\end{align*}	
We define the associated semi-norm as
\begin{align*}
\|u\|_{\dot{H}^s(\R^n)}:=\|\F^{-1} \{ |\xi|^s \hat{u} \} \|_{L^2(\R^n)}.
\end{align*}
Note that the norm $\|\cdot \|_{H^s(\R^n)}$ is equivalent to the norm $\|\cdot \|_{L^2(\R^n)}+\|\cdot\|_{\dot{H}^s(\R^n)}$.	For a more detailed introduction to fractional Sobolev spaces and related results, we refer the readers to \cite{di2012hitchhiks} and \cite{mclean2000strongly}.
	
	Since we will use this for our drift fields, we also recall the $L^p$ based fractional Sobolev spaces: we set $\|u\|_{W^{s,p}(\R^n)} := \|\langle D \rangle^s u \|_{L^p(\R^n)}$, where $\langle \xi \rangle = (1+|\xi|^2)^{1/2}$ and $m(D) u = \mathcal{F}^{-1}\{m(\xi) \widehat{u}(\xi)\}$ for $m\in C^{\infty}(\R^n)$ such that $m$ and all its derivatives are polynomially bounded, and $u$ is a tempered distribution. For an open set $\mathcal{O}\subset \R^n$ and $p>1$, we then define the space $W^{s,p}(\mathcal{O})$ by
	\begin{align*}
	W^{s,p}(\mathcal{O}) = \{u|_{\mathcal{O}}; \ u \in W^{s,p}(\R^n)\}.
	\end{align*}
	This is equipped with the associated norm
	\begin{align*}
	\|u\|_{W^{s,p}(\mathcal{O})} = \inf\{\|w\|_{W^{s,p}(\R^n)}; \ w \in W^{s,p}(\R^n), \ w|_{\mathcal{O}} = u\}.
	\end{align*}
	We also define 
    \begin{align*}	
	W^{s,p}_0(\mathcal{O}):= \mbox{ closure of } C_c^{\infty}(\mathcal{O}) \mbox{ in } W^{s,p}(\mathcal{O}).
	\end{align*}
	In the sequel, we will only use these more general $W^{s,p}$ function spaces to quantify the size of the drift field $b$.

We conclude this section by recalling a fractional Poincar\'e type inequality:
	
	\begin{lem}
		\label{lem:frac_Poinc}
		Let $n\geq 1$ and let $s\in (0,1)$. Assume that $\Omega \subset \R^n$ is open and bounded. Then, there exists a constant $C>0$ such that
		\begin{align*}
		\|v\|_{L^2(\Omega)} \leq C (\diam(\Omega))^{s} \|(-\D)^{s/2} v\|_{L^2(\R^n)} \mbox{ for } v \in \widetilde{H}^s(\Omega),
		\end{align*}
     where  $\diam(\Omega)$ denotes the diameter of $\Omega$.
	\end{lem}

	We present the proof for self-containedness but follow the idea from the appendix in \cite{RTZ18}.
	
	\begin{proof}[Proof of Lemma \ref{lem:frac_Poinc}]
		We first assume that $v\in C_c^{\infty}(\Omega)$. The result will then follow by density of $C_c^{\infty}(\Omega)$ in $\widetilde{H}^s(\Omega)$.
		Let $x\in \Omega$ be arbitrary and let $x' = x + 2 \diam(\Omega) \frac{x}{|x|}$. Let further $\widetilde{v}$ be the Caffarelli-Silvestre \cite{caffarelli2007extension} extension of $v$, i.e. let $\widetilde{v}$ be the solution of 
		\begin{align*}
		\nabla \cdot x_{n+1}^{1-2s} \nabla \widetilde{v} & = 0 \mbox{ in } \R^{n+1}_+,\\
		\widetilde{v} & = v \mbox{ on } \R^n \times \{0\}.
		\end{align*}
		Then the fundamental theorem of calculus and the support condition for $v$ imply that for any $r\in (\diam(\Omega),\infty)$ 
		\begin{align*}
		|v(x)| & = |\widetilde{v}(x,0)|
		\leq |\widetilde{v}(x,r)| + \int\limits_{0}^{r} |\p_{n+1} \widetilde{v}(x,t)| dt\\
		& \leq |\widetilde{v}(x',r)| + \int\limits_{0}^{1}|\nabla' \widetilde{v}(tx + (1-t)x',r) |  |x-x'|  dt + \int\limits_{0}^{r} |\p_{n+1} \widetilde{v}(x,t)| dt\\
		&\leq \int\limits_{0}^r |\p_{n+1} \widetilde{v}(x',t)|+ \int\limits_{0}^{r} |\p_{n+1} \widetilde{v}(x,t)| dt + \int\limits_{0}^{1}|\nabla' \widetilde{v}(tx + (1-t)x',r) |  |x-x'|  dt.
		\end{align*}
		Next we insert the weights $t^{\frac{1-2s}{2}}$, apply Hölder's inequality,  estimate $|x'-x| \leq Cr$ and fix $r \in (r_0, 2r_0)$, where $r_0 = \diam(\Omega)$:
		
		\begin{align*}
		|v(x)|& \leq  \int\limits_{0}^r t^{\frac{2s-1}{2}}t^{\frac{1-2s}{2}}|\p_{n+1} \widetilde{v}(x',t)| dt + \int\limits_{0}^{r}  t^{\frac{2s-1}{2}}t^{\frac{1-2s}{2}} |\p_{n+1} \widetilde{v}(x,t)| dt \\
		& \quad + \int\limits_{0}^{1}r^{\frac{1-2s}{2}} |\nabla' \widetilde{v}(tx + (1-t)x',r) |  |x-x'|  r^{\frac{2s-1}{2}}  dt\\
		& \leq Cr^s\left( \left(\int\limits_{0}^{r} x_{n+1}^{1-2s} |\p_{n+1} \widetilde{v}(x',x_{n+1})|^2 dx_{n+1} \right)^{\frac{1}{2}} + \left( \int\limits_{0}^r x_{n+1}^{1-2s} |\p_{n+1} \widetilde{v}(x,x_{n+1})|^2 dx_{n+1} \right)^{\frac{1}{2}} \right.\\
		& \quad \left. + r^{\frac{1}{2}}\left( \int\limits_{0}^1 r^{1-2s} |\nabla' \widetilde{v}(t x + (1-t)x',r)|^2 dt \right)^{\frac{1}{2}} \right)\\
		& \leq Cr^s_0\left( \left(\int\limits_{0}^{2 r_0} x_{n+1}^{1-2s} |\p_{n+1} \widetilde{v}(x',x_{n+1})|^2 dx_{n+1} \right)^{\frac{1}{2}} + \left( \int\limits_{0}^{2 r_0} x_{n+1}^{1-2s} |\p_{n+1} \widetilde{v}(x,x_{n+1})|^2 dx_{n+1} \right)^{\frac{1}{2}} \right.\\
		& \quad \left. +  r^{\frac{1}{2}}_0\left( \int\limits_{0}^1 r^{1-2s} |\nabla' \widetilde{v}(t x + (1-t)x',r)|^2 dt \right)^{\frac{1}{2}} \right).
		\end{align*}

		Here the constant $C>0$ in particular depends on $s$.
		Next we square the estimate and integrate it in the normal direction in the interval $r\in (r_0/2, 2r_0)$. This yields 
		
		\begin{align*}
		|v(x)|^2 
		&\leq 
		Cr^{2s}_0\left( \int\limits_{0}^{2 r_0} x_{n+1}^{1-2s} |\p_{n+1} \widetilde{v}(x',x_{n+1})|^2 dx_{n+1}  + \int\limits_{0}^{2 r_0} x_{n+1}^{1-2s} |\p_{n+1} \widetilde{v}(x,x_{n+1})|^2 dx_{n+1}   \right.\\
		& \quad \left. + 
\int\limits_{0}^1 \int\limits_{0}^{2r_0} x_{n+1}^{1-2s} |\nabla' \widetilde{v}(t x + (1-t)x',x_{n+1})|^2 dt dx_{n+1}	\right).
		\end{align*}
		Finally, integrating in tangential directions and using the support condition for $v$, we obtain
		\begin{align*}
		\|v\|_{L^2(\Omega)}^2 \leq C r_0^{2s} \|x_{n+1}^{\frac{1-2s}{2}} \nabla \widetilde{v}\|_{L^2(\R^{n+1}_+)}^2.
		\end{align*}
		Since by the work of Caffarelli-Silvestre \cite{caffarelli2007extension} we have 
		\begin{align*}
		\|x_{n+1}^{\frac{1-2s}{2}} \nabla \widetilde{v}\|_{L^2(\R^{n+1}_+)}^2 = \|v\|_{\dot{H}^{s}(\R^n)},
		\end{align*} 
	and since for $v\in C_c^{\infty}(\Omega)$ we have (in the sense of norm equivalences)
\begin{align*}
\|(-\D)^{s/2} v\|_{L^2(\R^n)} \sim \|v\|_{\dot{H}^{s}(\Omega)}  ,
\end{align*}	
this concludes the proof.
	\end{proof}
	
	\subsection{Well-posedness}
In this section we discuss the well-posedness of the equation \eqref{eq:main} and its dual equation.
	
To this end, let $\Omega\subset \R^n$ be a bounded domain (open and connected), $b\in W^{1-s,\infty}(\Omega)^n$ be a drift coefficient, $c\in L^\infty(\Omega)$ a potential and let $\frac{1}{2}<s<1$ be a constant. For $F\in (\widetilde H^s(\Omega))^{\ast}$ (the dual space of $ H^s(\Omega)$), $f\in H^{s}(\mathbb{R}^n)$, let us consider the following Dirichlet problem 
	\begin{align}\label{Dirichlet problem}
	\begin{split}
	((-\Delta)^s+b\cdot \nabla +c)u=F & \text{ in }\Omega,\\
	u-f & \in \widetilde{H}^s(\Omega).
	\end{split}
	\end{align}
	Given an arbitrary open set $\mathcal O\subset \R^n$ and $v,w\in L^2(\mathcal O)$, we use the notation 
	$$
	(v,w)_{\mathcal O}:=\int_{\mathcal O}vw dx.
	$$
	For $v,w \in C^{\infty}_c(\R^n)$ we define the bilinear form $B_{b,c}(\cdot,\cdot)$ by 
	\begin{align}\label{bilinear form}
	B_{b,c}(v,w):=((-\Delta)^{s/2}v,(-\Delta)^{s/2}w)_{\R^n}+(b\cdot\nabla v,w)_{\Omega}+(cv,w)_\Omega.
	\end{align}
	Notice that the bilinear form $B_{b,c}(\cdot, \cdot)$ is \emph{not} symmetric, so we also introduce the \emph{adjoint} bilinear form as 
	\begin{align}\label{adjoint bilinear form}
	B_{b,c}^{\ast}(v^{\ast},w^{\ast}):=((-\Delta)^{s/2}v^{\ast},(-\Delta)^{s/2}w^{\ast})_{\R^n} +(b v^{\ast}, \nabla w^{\ast} )_\Omega +(c v^{\ast},w^{\ast})_\Omega,
	\end{align}
	for $v^{\ast},w^{\ast} \in C^{\infty}_c(\R^n)$.
	
We remark that the term ``adjoint" is used with a slight abuse of notation here, e.g. as we did not specify the underlying function spaces. We however think of the adjoint bilinear form \eqref{adjoint bilinear form}, as the bilinear form associated with the \emph{adjoint exterior value problem} 
\begin{align}\label{adjoint Dirichlet problem}
\begin{split}
(-\Delta)^s u^{\ast} - \nabla \cdot (b u^{\ast})+cu^{\ast}=F^{\ast} &\text{ in }\Omega, \\ u^{\ast} - f^{\ast} &\in \widetilde{H}^s(\Omega),
\end{split}
\end{align}
for some suitable source $F^{\ast}$ and exterior datum $f^{\ast}$.

\begin{rmk}
\label{rmk:bilin_vs_eq}
We further stress that there is a slight discrepancy between the bilinear form \eqref{adjoint bilinear form} and the adjoint Dirichlet problem \eqref{adjoint Dirichlet problem} in that we have ignored the boundary contribution originating from the integration by parts of \eqref{adjoint Dirichlet problem} in the definition of the bilinear form \eqref{bilinear form}. In the sequel, this will for instance be reflected in the (symmetry) properties of the operator $\Lambda_{b,c}^{\ast}$, c.f. Lemma \ref{lem:symmetry}. More precisely, the DN map associated with the equation \eqref{adjoint Dirichlet problem} would contain a boundary contribution on $\partial \Omega$ if $b$ (or $u^{\ast}$) does not vanish there. This is a consequence of an integration by parts which is used to obtain the weak form of the DN map associated with the equation \eqref{adjoint Dirichlet problem}. For the operator $\Lambda_{b,c}^{\ast}$ which is defined through the ``adjoint bilinear form" \eqref{adjoint bilinear form} this contribution is \emph{not} present, as the integration by parts has already been carried out.
\end{rmk}

Before continuing in our discussion, we observe that under our regularity assumptions the bilinear forms \eqref{bilinear form}, \eqref{adjoint bilinear form} are well-defined on $H^{s}(\R^n)$, for $\frac{1}{2}<s<1$.

\begin{lem}[Boundedness of bilinear forms]
\label{lem:bilin_form}
		For $n\geq 1$, let $\Omega \subset \R^n$ be a bounded Lipschitz domain and $\frac{1}{2}<s<1$. Assume that $b\in W^{1-s,\infty}(\Omega)^n$, $c\in L^{\infty}(\Omega)$. Let $B_{b,c}(\cdot, \cdot)$ and $B^{\ast}_{b,c}(\cdot,\cdot)$ be the bilinear and the adjoint bilinear form defined by \eqref{bilinear form} and \eqref{adjoint bilinear form}, respectively.
		Then, $B_{b,c}(\cdot,\cdot)$ and $B_{b,c}^{\ast}(\cdot,\cdot)$ extend as bounded bilinear forms on $H^s(\R^n)\times H^{s}(\R^n)$. 
\end{lem}

\begin{proof}
It suffices to discuss the extension of $B_{b,c}(\cdot,\cdot)$ as the argument for $B_{b,c}^{\ast}(\cdot,\cdot)$ is analogous. First, we directly have
\begin{align}\label{bounded estimate 1}
&\notag \quad \left|((-\Delta)^{s/2}v,(-\Delta)^{s/2}w)_{\R^n}\right|+\left|(cv,w)_\Omega\right| \\
&\notag \leq C(1+\|c\|_{L^{\infty}(\Omega)}) \left(\|v\|_{\dot{H}^{s}(\R^n)}\|w\|_{\dot{H}^{s}(\R^n)} + \|v\|_{L^2(\Omega)}\|w\|_{L^2(\Omega)}\right)\\
&\leq C(1+\|c\|_{L^{\infty}(\Omega)}) \|v\|_{H^{s}(\R^n)}\|w\|_{H^{s}(\R^n)},
\end{align}
for some constant $C>0$ independent of $v,w$. It remains to discuss the contribution of the drift term.
For this we note that as $1-s \in (0,\frac{1}{2})$ and as $\Omega$ is a Lipschitz domain, we have $\widetilde{H}^{1-s}(\Omega) = H^{1-s}(\Omega)$ and $\|w\|_{\widetilde{H}^{1-s}(\Omega)} \leq C \|w\|_{H^{1-s}(\R^n)}$.
Next, we choose $B \in W^{1-s,\infty}(\R^n)$ such that
		\begin{align*}
		B|_{\Omega} = b, \ \|B\|_{W^{1-s,\infty}(\R^n)} \leq 2\|b\|_{W^{1-s,\infty}(\Omega)}.
		\end{align*} 
		and estimate
		\begin{align}\label{estimate of drift term_a}
		\begin{split} 
		\left|\int_\Omega w (b\cdot \nabla v)   dx\right|
		& \leq  \|b w\|_{\widetilde{H}^{1-s}(\Omega)} \|  \nabla v\|_{H^{s-1}(\Omega)} \\
		& \leq C \| B w \|_{H^{1-s}(\R^n)} \|  \nabla v\|_{H^{s-1}(\R^n)}\\
		& \leq C \|w \|_{H^{1-s}(\R^n)} \| b \|_{W^{1-s,\infty}(\Omega)} \| \nabla v\|_{H^{s-1}(\R^n)} \\
		& \leq C \|w \|_{H^{1-s}(\R^n)} \| b \|_{W^{1-s,\infty}(\Omega)} \| v\|_{H^{s}(\R^n)}\\
		& \leq C \|w \|_{H^{s}(\R^n)} \| b \|_{W^{1-s,\infty}(\Omega)} \| v\|_{H^{s}(\R^n)}.
		\end{split}
		\end{align}
Here we used the Kato-Ponce inequality \cite{GO14} in order to obtain a suitable multiplier estimate
		\begin{align}\label{multiplier estimate}
		\begin{split}
		\|B w\|_{H^{1-s}(\R^n)} 
		&\leq C \|J^{1-s}(B w)\|_{L^2(\R^n)} \\
		&\leq C \left(\|B\|_{L^{\infty}(\R^n)} \|J^{1-s} w \|_{L^2(\R^n)} + \|J^{1-s}B \|_{L^{\infty}(\R^n)} \|w \|_{L^2(\R^n)}\right)\\
		& \leq C \|B\|_{W^{1-s,\infty}(\R^n)} \|w\|_{H^{1-s}(\R^n)}\\
		& \leq 2C \|b\|_{W^{1-s,\infty}(\Omega)} \|w\|_{H^{1-s}(\R^n)}\\
		& \leq 2C \|b\|_{W^{1-s,\infty}(\Omega)} \|w\|_{H^{s}(\R^n)},
		\end{split}
		\end{align}
		where $J^{1-s}:=(\Delta -1)^{\frac{1-s}{2}}$ and $0<1-s< \frac{1}{2}$. Finally, combining \eqref{bounded estimate 1}, \eqref{estimate of drift term_a} and \eqref{multiplier estimate}, one has the desired result 
		\begin{align}\label{boundedness of bilinear form in H^s}
			\left|B_{b,c}(v,w)\right|\leq C\|v\|_{H^s(\R^n)}\|w\|_{H^s(\R^n)}, 
		\end{align}
		for $\frac{1}{2}<s<1$ and for some constant $C>0$ independent of $v,w$. A similar argument also yields the boundedness for the adjoint bilinear form, which reads
		\begin{align}
		\label{boundedness of adjoint bilinear form in H^s}
		\left|B_{b,c}^\ast(v^\ast,w^\ast)\right|\leq C\|v^\ast\|_{H^s(\R^n)}\|w^\ast\|_{H^s(\R^n)}. 
		\end{align}
\end{proof}
	
\begin{rmk} For the above well-definedness argument, we point out the following observations:
\label{rmk:reg_1}
\begin{itemize}
\item[(a)] The requirement that $\Omega$ is a bounded Lipschitz domain only entered in the identity $\widetilde{H}^{1-s}(\Omega) = H^{1-s}(\Omega)$. If this identity holds for a less regular domain $\Omega$ (e.g. if $\partial \Omega$ is Hölder continuous with sufficiently high Hölder exponent depending on $s$) the well-definedness of the bilinear forms \eqref{bilinear form}, \eqref{adjoint bilinear form} persists for this domain. In order to avoid technicalities, we do not address this issue in the sequel, but will always assume that $\Omega$ is Lipschitz. Nevertheless, we phrase our results (e.g. the Dirichlet-to-Neumann map in Definition \ref{defi:DtN}) such that they remain valid for a more general class of domains which satisfy the condition $\widetilde{H}^{1-s}(\Omega) = H^{1-s}(\Omega)$.

\item[(b)] As an alternative condition to imposing Lipschitz regularity, we might also have asked for $b\in W^{1-s,\infty}_0(\Omega)^n$. As the drift field $b$ however is one of the main objects of interest in our argument, we opted for rather imposing regularity on $\Omega$ than on restricting the class of admissible drift fields. 
\end{itemize}
\end{rmk}	

	We in particular notice that by the above definitions, \eqref{bilinear form}, \eqref{adjoint bilinear form} and the estimates \eqref{boundedness of bilinear form in H^s} and \eqref{boundedness of adjoint bilinear form in H^s}, we obtain 
	\begin{align}
	\label{eq:non_symm}
	B_{b,c}(u,w) = B_{b,c}^{\ast}(w,u), \text{ for any }u,w\in H^s(\R^n).
	\end{align}
	
	With these bilinear forms at hand, we define $u\in H^s(\R^n)$ with $u-f\in \widetilde{H}^s(\Omega)$ to be a solution of \eqref{Dirichlet problem} if for all $w\in \widetilde{H}^{s}(\R^n)$ we have $B_{b,c}(u,w)=F(w)$. 
	A solution to the adjoint problem is defined analogously using $B^{\ast}_{b,c}(\cdot,\cdot)$.
	
	Relying on energy estimates, we can prove the existence and uniqueness of solutions to \eqref{Dirichlet problem} and \eqref{adjoint Dirichlet problem}, outside of a discrete set of eigenvalues:
	
	\begin{prop}
		\label{prop:well_posed}
		Let $\Omega \subset \R^n$, $n\geq 1$, be a bounded Lipschitz domain and $\frac{1}{2}<s<1$. Assume that $b\in W^{1-s,\infty}(\Omega)^n$, $c\in L^{\infty}(\Omega)$. Let $B_{b,c}(\cdot, \cdot)$ and $B^{\ast}_{b,c}(\cdot,\cdot)$ be the bilinear and the adjoint bilinear form defined by \eqref{bilinear form} and \eqref{adjoint bilinear form}, respectively. Then the following properties hold: 
		\begin{itemize}
			\item[(a)] There exists a countable set $\Sigma = \{ \lambda_j \}_{j=1}^{\infty} \subset \R$ with $\lambda_1 \leq \lambda_2 \leq \cdots \rightarrow \infty$ such that if $\lambda \in \R \setminus \Sigma$, for any $f\in H^{s}(\mathbb{R}^n)$ and $F\in (\widetilde{H}^s(\Omega))^{\ast}$ there exists a unique function $u\in H^{s}(\R^n)$ such that
			\begin{align*}
			B_{b,c}(u,w) - \lambda(u,w)_{\Omega} = F(w) \mbox{ for all } w \in \widetilde{H}^s(\Omega) \mbox{ and } \ u-f \in \widetilde{H}^{s}(\Omega).
			\end{align*} 
			We then have
			\begin{align*}
			\|u\|_{H^s(\R^n)} \leq C(\|f\|_{H^{s}(\R^n)} + \|F\|_{(H^{s}(\R^n))^{\ast}}),
			\end{align*}
			where the constant $C>0$ depends on $n,s, \|b\|_{W^{1-s,\infty}(\Omega)}, \|c\|_{L^{\infty}(\Omega)}, \lambda$.
			Moreover, $\Sigma \subset (-\|c\|_{L^{\infty}(\Omega)}-C \|b\|_{W^{1-s,\infty}(\Omega)}, \infty)$ where the constant $C>0$ is sufficiently large and only depends on $s,n$.
			\item[(b)] Similarly, for the adjoint bilinear form \eqref{adjoint bilinear form}, there exists a countable set $\Sigma ^{\ast}\subset \R$ such that if $\lambda^{\ast}\in \R \setminus \Sigma ^{\ast}$, for all $f^* \in H^s(\mathbb{R}^n)$ there is a unique $u^* \in H^s(\mathbb{R}^n)$ with
			\begin{align*}
			B_{b,c}^{\ast}(u^{\ast},w^{\ast}) - \lambda^{\ast}(u^{\ast},w^{\ast})_{\Omega} = F^{\ast}(w^{\ast}), \text{ for all }w^{\ast}\in \widetilde{H}^s(\Omega) \text{ and } u^{\ast}-f^{\ast} \in \widetilde{H}^{s}(\Omega).
			\end{align*} 
		\end{itemize}
	\end{prop}
	
	\begin{rmk}
		\label{rmk:reg_b}
		We point out that the regularity of $b \in W^{1-s,\infty}(\Omega)^n$ is determined by the properties of multipliers in the function space $H^{1-s}(\Omega)$ for $s\in (\frac{1}{2},1)$ (through the application of the Kato-Ponce inequality, c.f. the proof of the well-posedness result below and \cite{GO14}). This is consistent with the fact that for $s=1$ it is easily seen that $b\in L^{\infty}(\Omega)$ is a sufficient condition for the well-posedness of the problem from Proposition \ref{prop:well_posed}.
	\end{rmk}

	\begin{rmk}
		\label{rmk:well_posed}
		For $\Omega$ Lipschitz, we recall that $\widetilde{H}^s(\Omega) = H^{s}_{\overline{\Omega}}$. In particular the above well-posedness theory then also transfers to these functions spaces immediately.
	\end{rmk}

\begin{rmk}
\label{rmk:reg_fct_spaces}
We emphasize that we did not attempt to optimize the function spaces for the drift and the potential contributions here. It is an interesting question which we plan to investigate in future work under which regularity conditions the properties of the inverse problem under investigation persist.
\end{rmk}	
	
	\begin{proof}[Proof of Proposition \ref{prop:well_posed}]
		We first give the proof of (a); it will follow from the spectral theorem for compact operators.
		Indeed, we first assume that $f=0$, which can always be achieved by considering $v=u-f\in \widetilde{H}^s(\Omega)$. Arguing similarly as in \eqref{multiplier estimate} for the resulting drift term, we note that this yields an appropriately modified functional $\widetilde F \in (\widetilde{H}^{s}(\Omega))^{\ast}$. In the sequel, we only consider the function $v$. Since the boundedness of the bilinear form was already proven by Lemma \ref{lem:bilin_form}, we only need to prove the coercivity of $B_{b,c}(\cdot,\cdot)$.
		
 To this end, we first note that for $ \phi \in C_c^{\infty}(\R^n)$ the following  interpolation estimate holds 
		\begin{align}\label{interpolation estimate}
		\| \phi\|_{H^{1-s}(\R^n)} \leq C \|\phi \|_{L^2(\R^n)}^{\frac{2s-1}{s}} \|\phi \|_{H^{s}(\R^n)}^{\frac{1-s}{s}} \mbox{ for } \frac{1}{2}<s<1.
		\end{align}
	    By virtue of Remark \ref{rmk on Sobolev} and the interpolation estimate \eqref{interpolation estimate}, one has for $\phi, v \in C_c^{\infty}(\R^n)$, 
		\begin{align}\label{estimate of drift term}
		\begin{split} 
		\left|\int_\Omega \phi (\nabla v\cdot  b)   dx\right|
		& \leq  \|b \phi\|_{\widetilde{H}^{1-s}(\Omega)} \|  \nabla v\|_{H^{s-1}(\Omega)} \\
		& \leq C \| b \phi \|_{H^{1-s}(\R^n)} \|  \nabla v\|_{H^{s-1}(\R^n)}\\
		& \leq C \|\phi \|_{H^{1-s}(\R^n)} \| b \|_{W^{1-s,\infty}(\Omega)} \| \nabla v\|_{H^{s-1}(\R^n)} \\
		& \leq C \|\phi \|_{H^{1-s}(\R^n)} \| b \|_{W^{1-s,\infty}(\Omega)} \| v\|_{H^{s}(\R^n)}\\
		& \leq C \|v \|_{H^{s}(\R^n)} \| b \|_{W^{1-s,\infty}(\Omega)} \|\phi\|_{L^2(\R^n)}^{\frac{2s-1}{s}} \| \phi\|_{H^{s}(\R^n)}^{\frac{1-s}{s}}\\
		& \leq C_{\epsilon} \|b\|_{W^{1-s,\infty}(\Omega)} \|v \|_{H^{s}(\R^n)} \|\phi\|_{L^{2}(\R^n)} + \epsilon  \|v \|_{H^{s}(\R^n)} \|\phi\|_{H^{s}(\R^n)},
		\end{split}
		\end{align}	
which holds for any $\epsilon>0$ and for (generic) constants $C,C_\epsilon >0$ which are independent of $v$ and $\phi$.	
Here we have utilized Young's inequality in the last line; and used the Kato-Ponce inequality as in \eqref{estimate of drift term_a}. Notice that by the density of $C_c^{\infty}(\Omega)$ in $\widetilde{H}^s(\Omega)$, the estimate \eqref{estimate of drift term} also extends to $v, \phi \in \widetilde{H}^s(\Omega)$.
		A matching coercivity estimate follows from the Poincar\'e inequality (c.f. Lemma \ref{lem:frac_Poinc})
		\begin{align*}
		\|v\|_{L^2(\R^n)}  \leq C \|(-\D)^{s/2} v\|_{L^2(\R^n)}, \ v \in \widetilde{H}^s(\Omega),
		\end{align*}
		in combination with \eqref{estimate of drift term} and Young's inequality:
		
		\begin{align*}
		B_{b,c}(v,v) &\geq \|(-\D)^{s/2} v\|_{L^2(\R^n)}^2 - \|c\|_{L^{\infty}(\R^n)}\|v\|_{L^2(\R^n)}^2  - \left| \int\limits_{\Omega} v (\nabla v \cdot b) dx \right|\\
		& \geq \|(-\D)^{s/2} v\|_{L^2(\R^n)}^2 - \|c\|_{L^{\infty}(\R^n)}\|v\|_{L^2(\R^n)}^2  \\
		&\quad - C_{\epsilon} \|b\|_{W^{1-s,\infty}(\Omega)} \|v\|_{L^2(\R^n)}^2 - \epsilon \|v\|_{H^s(\R^n)}^2 \\
		& \geq c_0(\|(-\D)^{s/2} v\|_{L^2(\R^n)}^2  + \|v\|_{L^2(\R^n)}^2) - \|c\|_{L^{\infty}(\R^n)}\|v\|_{L^2(\R^n)}^2  \\
		& \quad - C_{\epsilon} \|b\|_{W^{1-s,\infty}(\Omega)} \|v\|_{L^2(\R^n)}^2 - \epsilon \|v\|_{H^s(\R^n)}^2 \\
		& \geq \frac{c_0}{2}\|v\|_{H^{s}(\R^n)}^2  - \|c\|_{L^{\infty}(\R^n)}\|v\|_{L^2(\R^n)}^2  - C_{c_0} \|b\|_{W^{1-s,\infty}(\Omega)} \|v\|_{L^2(\R^n)}^2,
		\end{align*}
		for some constant $c_0>0$ and $C>0$ independent of $v$.
		Here we have chosen $\epsilon = \frac{c_0}{2}$ in the above calculation.
		
		As a consequence, for $\mu = C_{c_0}\|b\|_{W^{1-s,\infty}(\Omega)} + \|c\|_{L^{\infty}(\Omega)} \geq 0$ we have
		\begin{align*}
		B_{b,c}(v,v) + \mu (v,v)_{L^2(\R^n)} \geq \frac{c_0}{2} \|v\|_{H^s(\R^n)}^2, \text{ for any }v\in \widetilde{H}^s(\Omega).
		\end{align*}
		Thus, for $\mu$ as above, $B_{b,c}(\cdot, \cdot) + \mu (\cdot,\cdot)_{L^2(\Omega)}$ is a scalar product on $\widetilde{H}^s(\Omega)$. Applying the Riesz representation theorem, we then infer that for all $\widetilde F\in (\widetilde{H}^s(\Omega))^{\ast}$ there exists $v\in \widetilde{H}^s(\Omega)$ such that
		\begin{align*}
		B_{b,c}(v,\phi) + \mu (v,\phi)_{L^2(\Omega)} & =\widetilde  F(\phi) \mbox{ for all } \phi \in \widetilde{H}^s(\Omega).
		\end{align*}
		In fact, we may write $v = G_{\mu}(\widetilde F) $, where $G_{\mu}$ is a bounded, linear operator from $(\widetilde{H}^s(\Omega))^{\ast} \rightarrow \widetilde{H}^s(\Omega)$. By the compact Sobolev embedding of $\widetilde{H}^s(\Omega) \hookrightarrow L^2(\Omega)$ the operator $G_{\mu}$ is a compact operator, if interpreted as a map from $L^2(\Omega)$ to $L^2(\Omega)$. 
		
		Since the equation
		\begin{align}
		\label{eq:eq_to_solve}
		B_{b,c}(v,\cdot) - \lambda(v,\cdot) = \widetilde F(\cdot) \mbox{ on } \widetilde{H}^s(\Omega)
		\end{align}
		is equivalent to the equation $v= G_{\mu}((\mu+\lambda)v + \widetilde F)$, we may now invoke the spectral theorem for compact operators to infer that there exists a countable, decreasing sequence $\{\gamma_j\}_{j=1}^{\infty} \subset [0,\infty)$ with $\gamma_j \rightarrow 0$ such that if $\gamma \notin \{\gamma_j\}_{j=1}^{\infty}$ the operator $G-\gamma Id$ is invertible. In particular, we may rewrite $\gamma_j = (\lambda_j+\mu)^{-1}$ for a countable, increasing sequence with $\lambda_j \rightarrow \infty$. Recalling that \eqref{eq:eq_to_solve} can be written as a corresponding operator equation, then concludes the proof of the solvability result for $B_{b,c}(\cdot,\cdot)$. 
		
		For (b), the proof is similar to (a). As in the previous arguments, one only needs to check the boundedness of $\left| \int_{\Omega} bv^{\ast}\cdot \nabla \phi^{\ast}dx \right|$, where $v^{\ast}=u^{\ast}-f^{\ast}\in \widetilde H^s(\Omega)$ and $\phi^{\ast}\in \widetilde H^s(\Omega)$ is an arbitrary test function. Here $u^{\ast}$ and $f^{\ast}$ are the functions which appeared in the equation \eqref{adjoint Dirichlet problem}. In particular, by using the Kato-Ponce inequality, the interpolation inequality \eqref{interpolation estimate} and Young's inequality again, for any $\epsilon>0$, we have 
		\begin{align*}
		\left|\int_\Omega bv^{\ast}\cdot \nabla \phi^{\ast} dx\right|\leq C_{\epsilon} \|b\|_{W^{1-s,\infty}(\Omega)} \|\phi^{\ast}\|_{H^{s}(\R^n)} \|v^{\ast}\|_{L^{2}(\R^n)} + \epsilon  \|\phi^{\ast}\|_{H^{s}(\R^n)} \|v^{\ast}\|_{H^{s}(\R^n)},
		\end{align*}
		for some positive constants $C_\epsilon$ independent of $v^{\ast}$ and $\phi^{\ast}$. The remainder of the proof follows along the same lines as the proof of (a). This completes the proof of the proposition.
\end{proof}

	\begin{rmk} 
\label{rmk:conditions_s}	
We point out that:
		\begin{itemize}
			\item[(a)] Relying on our assumption \eqref{eigenvalue condition}, and combining this with the results of Proposition \ref{prop:well_posed} and the Fredholm alternative, then implies the well-posedness of our problem \eqref{eq:main}. In particular, the Fredholm alternative also yields that \eqref{eigenvalue condition} is equivalent to the following condition (for example, see  \cite[Theorem 2.27]{mclean2000strongly} or \cite[Chapter 5.3]{gilbarg2015elliptic})
\begin{align*}
&w^{\ast}\in H^s(\R^n) \text{ is a solution of }(-\Delta)^s w^{\ast}-\nabla \cdot  (b w^{\ast} ) +cw^{\ast}=0\text{ in }\Omega \mbox{ with } w^{\ast} = 0 \mbox{ in } \Omega_e, \\
& \text{then we have }w^{\ast} \equiv 0.
\end{align*}

			\item[(b)] The proofs of Lemma \ref{lem:bilin_form} and Proposition \ref{prop:well_posed} relied on the fact that the fractional Laplacian was the leading order operator of our equations which was ensured by the condition $\frac{1}{2}<s<1$. In particular, the above arguments and results do not persist in the regime $0<s\leq \frac{1}{2}$.
		\end{itemize}
	\end{rmk}

	Now, let $\Omega\subset \R^n$ be a bounded open set and $\frac{1}{2}<s<1$. We consider the Dirichlet problem \eqref{Dirichlet problem} with a zero source function, i.e., 
	\begin{align}\label{zero Dirichlet problem}
	\begin{split}
	((-\Delta)^s+b\cdot \nabla +c)u=0 & \text{ in }\Omega,\\
	u-f & \in \widetilde{H}^s(\Omega),
	\end{split}
	\end{align}
	for some $f\in H^s(\Omega_e)$.
	Recall that the eigenvalue condition \eqref{eigenvalue condition} in combination with Proposition \ref{prop:well_posed} shows that there exists a unique solution $u\in H^s(\R^n)$ of \eqref{zero Dirichlet problem}. We would like to point out that the solution $u$ here depends only on $f$ modulo $\widetilde{H}^s(\Omega)$. In effect, we emphasize that the solution of the Dirichlet problem \eqref{zero Dirichlet problem} is only affected by the exterior datum. 
	
	Therefore, following \cite{ghosh2016calder} and \cite{ghosh2017calderon}, we introduce the quotient space 
$$
\mathbb{X} = H^s(\mathbb{R}^n) / \widetilde{H}^s(\Omega).
$$	
	We will denote the equivalence class of $f \in H^s(\mathbb{R}^n)$ by $[f]$. We also recall that for $\Omega$ a bounded, open Lipschitz set, we have $\mathbb{X}=H^s(\Omega_e)$ by Remark \ref{rmk on Sobolev}. Based on the well-posedness of \eqref{eq:main}, one can define the DN map rigorously as follows.
	
	\begin{defi}[DN map]
		\label{defi:DtN}
		Let $\Omega \subset \R^n$, $n\geq 1$, be a bounded Lipschitz domain, let $s\in (\frac{1}{2},1)$ and let $b\in W^{1-s,\infty}(\Omega)^n$, 
$c\in L^{\infty}(\Omega)$. Let $B_{b,c}(\cdot,\cdot)$ be given as \eqref{bilinear form}. Then we define the Dirichlet-to-Neumann operator associated with the equation  \eqref{eq:main} as 
		\begin{align*}
		\Lambda_{b,c}: \mathbb{X} \rightarrow \mathbb{X}^{\ast},\ 
		(\Lambda_{b,c}[f], [g]) = B_{b,c}(u_f, g),
		\end{align*}
		where $f, g\in H^{s}(\mathbb{R}^n)$ and $u_f$ is the weak solution to \eqref{eq:main} with exterior datum $f$.
	\end{defi}
	
	We first remark that this is well-defined, since by the definition of a solution to \eqref{eq:main} we have $B_{b,c}(u_f,\widetilde{g})=B_{b,c}(u_f, \widetilde{g}+ \psi)$ for any $\psi \in \widetilde{H}^s(\Omega)$. Also $B_{b, c}(u_{f + \varphi}, g) = B_{b, c}(u_f, g)$ for $\varphi \in \widetilde{H}^s(\Omega)$, since $u_{f + \varphi} = u_f$ (by the uniqueness of solutions). Finally, the boundedness of the DN map $\Lambda_{b, c}: \mathbb{X} \to \mathbb{X}^*$  follows easily, once we have the multiplier estimate \eqref{multiplier estimate} and part (a) of Proposition \ref{prop:well_posed}.

\begin{rmk}
\label{rmk:DtN_char}	
	Furthermore, we remark that a formal calculation using Definition \ref{defi:DtN} yields
	\begin{align}
	(\Lambda_{b,c}[f],[g]) & =B_{b,c}(u_{f},g)\nonumber \\
	& =\int_{\mathbb{R}^{n}}(-\Delta)^{s/2}u_f(-\Delta)^{s/2}gdx+\int_\Omega b\cdot\nabla u_f gdx+\int_\Omega cu_f gdx\nonumber \\
	& =\int_{\Omega_{e}}g (-\Delta)^s u_f dx,\label{t15}
	\end{align}
	where $u_f$ is the weak solution to \eqref{eq:main} with exterior datum $f$.
	Then from \eqref{t15}, one can formally obtain that 
	\begin{equation*}
	\Lambda_{b,c}[f]=\left.(-\Delta)^s u_{f}\right|_{\Omega_{e}},
	\end{equation*}
	which gives evidence of \eqref{DN map}.
	We remark that in order to make this calculation rigorous, higher regularity assumptions have to be imposed. We refer to Section 3, Lemma 3.1, in \cite{ghosh2016calder} for details on this.
\end{rmk}

	We may also consider the adjoint problem with a zero source term:
	\begin{align}
	\label{eq:adjoint}
	\begin{split}
	(-\D)^s u^{\ast} - \nabla \cdot (b u^{\ast}) + c u^{\ast}  = 0 &\mbox{ in } \Omega,\\
	u^{\ast} - f &\in \widetilde{H}^s(\Omega),
	\end{split}
	\end{align}
	Then by using the eigenvalue condition \eqref{eigenvalue condition} together with Proposition \ref{prop:well_posed}, this problem has a unique solution $u^*$. With slight abuse of notation, one can analogously define the DN map $\Lambda_{b,c}^{\ast}$ of the adjoint bilinear form \eqref{adjoint bilinear form}, i.e.,
			\begin{align*}
		\Lambda_{b,c}^{\ast}: \mathbb{X} \rightarrow \mathbb{X}^{\ast},\ 
		(\Lambda_{b,c}^{\ast}[f], [g]) = B_{b,c}^{\ast}(u_f^\ast, g).
		\end{align*}
		Here $u_f^* \in H^s(\mathbb{R})^n$ is a weak solution of \eqref{eq:adjoint} with exterior data $f$.

Next, let $\Lambda'_{b,c}: \mathbb{X} \to \mathbb{X}^*$ be the adjoint operator with respect to the DN map $\Lambda_{b,c}$, i.e., the adjoint DN map $\Lambda'_{b,c}$ is defined via
$$
(\Lambda_{b,c}f,g) = (f, \Lambda'_{b,c}g), \text{ for }f,g\in H^s(\mathbb{R}^n).
$$
Then we further observe the following relation between the DN map associated with \eqref{eq:main} and the DN map of the adjoint bilinear form, which states that the adjoint DN map $\Lambda'_{b,c}$ and the DN map $\Lambda_{b,c}^\ast$ with respect to the adjoint equation coincide. In order to simplify notation, here and in the sequel, we drop the brackets $[\cdot]$ for the elements of $\mathbb{X}$.
	
	\begin{lem}
		\label{lem:symmetry}
		Let $\Omega \subset \R^n$, $n\geq 1$, be a bounded open Lipschitz set, $s\in (\frac{1}{2},1)$, $b\in W^{1-s,\infty}(\Omega)^n$,  
$c\in L^{\infty}(\Omega)$ and assume that \eqref{eigenvalue condition} holds. Let $B_{b,c}(\cdot,\cdot)$, $B_{b,c}^{\ast}(\cdot,\cdot)$ and $\Lambda_{b,c}, \Lambda_{b,c}^{\ast}$ be given as above. Then, for all $f,g \in \mathbb{X}$
		\begin{align}\label{eq: DN map and its adjoint}
		(\Lambda_{b,c}f,g) = (f, \Lambda_{b,c}^{\ast}g).
		\end{align}
	\end{lem}
	
	\begin{proof}
		The claim follows from the independence
		of the DN map and the adjoint DN map of the extension of $g$ into $\Omega$ in the duality pairing from Definition \ref{defi:DtN}. Let $u_f$ be a solution to \eqref{Dirichlet problem} with exterior data $f\in \mathbb{X}$ and let $u^{\ast}_g$ be a solution of the adjoint equation \eqref{eq:adjoint} with exterior data $g \in \mathbb{X}$. Then, by \eqref{eq:non_symm} and the definition of $\Lambda_{b,c}$, $\Lambda_{b,c}^{\ast}$
		\begin{align*}
		(\Lambda_{b,c}f,g)
		&= B_{b,c}(u_f, u^{\ast}_g)
		= B_{b,c}^{\ast}(u^{\ast}_g, u_f) 
		= (\Lambda_{b,c}^{\ast}g,f) =  (f,\Lambda_{b,c}^{\ast}g). \qedhere
		\end{align*}
	\end{proof}

	With this at hand, we seek to derive a corresponding Alessandrini type identity. It will play a key role in our uniqueness and stability arguments.
	
	\begin{lem}[Alessandrini identity]
		Let $\Omega \subset \R^n$, $n\geq 1$, be a bounded Lipschitz domain and $\frac{1}{2}<s<1$. Let $b_j\in W^{1-s,\infty}(\Omega)^n$ and $c_j\in L^\infty(\Omega)$ be such that \eqref{eigenvalue condition} holds for $j=1,2$. For any $f_1, f_2\in \mathbb{X}$, we have 
		\begin{align}\label{eq: Alessandrini}
		\left((\Lambda_{b_1,c_1}-\Lambda_{b_2,c_2})f_1,f_2\right)_{\mathbb{X}^{\ast}\times \mathbb{X}}=((b_1-b_2)\cdot \nabla u_1,u_2^{\ast})_\Omega+((c_1-c_2)u_1,u_2^{\ast})_\Omega,	
		\end{align}
		where $u_1 \in H^s(\R^n)$ is the solution to $((-\Delta)^s+b_1\cdot \nabla +c_1)u_1=0 $ in $\Omega$ and $u_2^{\ast}\in H^s(\R^n)$ is the solution to $(-\Delta)^s u_2^{\ast}-\nabla \cdot (b_2 u_2^{\ast}) + c_2 u_2^{\ast}=0$ in $\Omega$ with $u_1 - f_1 \in \widetilde{H}^s(\Omega)$ and $u_2 - f_2 \in \widetilde{H}^s(\Omega)$.
	\end{lem}
	\begin{proof}
		By \eqref{eq: DN map and its adjoint}, one has 
		\begin{align*}
		((\Lambda_{b_1,c_1}-\Lambda_{b_2,c_2})f_1,f_2)=&(\Lambda_{b_1,c_1}f_1,f_2)-(f_1,\Lambda_{b_2,c_2}^{\ast}f_2) \\
		=& B_{b_1,c_1}(u_1,u_2^{\ast})-B_{b_2,c_2}^{\ast}(u_2^{\ast},u_1)\\
		=& ((b_1-b_2)\cdot \nabla u_1,u_2^{\ast})_\Omega+((c_1-c_2)u_1,u_2^{\ast})_\Omega. \qedhere
		\end{align*}
	\end{proof}

	Last but not least, we define the Poisson operator associated with the equation \eqref{eq:main}:
	
	\begin{defi}
		\label{defi:Poisson}
		Let $\Omega \subset \R^n$, $n\geq 1$, be a bounded open Lipschitz set. Let $s\in (\frac{1}{2},1)$ and assume that \eqref{eigenvalue condition} holds. Then, the Poisson operator $P_{b,c}$ associated with \eqref{eq:main} is defined as
		\begin{align}\label{Poisson operators}
		P_{b,c}: \mathbb{X} \rightarrow H^{s}(\R^n), \ f \mapsto u_f,
		\end{align}
		where $u_f \in H^s(\mathbb{R}^n)$ with $u_f-f\in \widetilde{H}^s(\Omega)$ denotes the unique solution to \eqref{eq:main}. 
	\end{defi}

	\section{Approximation property}\label{sec:3}
	
	In this section we discuss the Runge approximation property for solutions to \eqref{eq:main}. To this end, we first recall the \emph{strong uniqueness} property for the fractional Laplacian, which was proved in \cite[Theorem 1.2]{ghosh2016calder}.
	\begin{prop}[Global weak unique continuation]\label{prop strong unique}
		Let $n\geq 1$, $s\in (0,1)$ and $u\in H^{-r}(\R^n)$ for some $r\in \R$. Assume that for some open set $W\subset \R^n$ we have
\begin{align*}		
		u=(-\Delta)^su =0 \mbox{ in } W.
\end{align*}		
Then $u\equiv 0$ in $\R^n$.
	\end{prop}

	By using the above strong uniqueness property and a duality argument, one can derive the following Runge approximation (which then immediately entails Theorem \ref{thm: Runge approximation} (a)).
	
	\begin{lem}
		\label{lem:Runge}
		For $n\geq 1$ and $\frac{1}{2}<s<1$, let $\Omega \subset \R^n$ be a bounded, open Lipschitz set, $b\in W^{1-s,\infty}(\Omega)^n$  
and $c\in L^\infty(\Omega)$ and let $P_{b,c}$ denote the Poisson operator from Definition \ref{defi:Poisson}. Further let $W\subset \Omega_e$ be an arbitrary open set. Then we have the following results: 
		\begin{itemize}
			\item[(a)] The set 
			\begin{align*}
			\mathcal{D}_1:=\left\{P_{b,c}f-f: \ f\in C^\infty _c (W)\right\}
			\end{align*}
			is dense in $L^2(\Omega)$.
			
			\item[(b)] The set 
			\begin{align*}
			\mathcal{D}_2:=\left\{P_{b,c}f-f: \ f\in C^\infty _c (W)\right\}
			\end{align*}
			is dense in $\widetilde{H}^s(\Omega)$.
		\end{itemize}
		
	\end{lem}
	\begin{proof}
		The proofs of (a) and (b) are similar and follow from a  duality argument. Hence, we only need to prove the case (b). First, it is easy to see that $\mathcal D_2 \subset \widetilde H^s(\Omega)$. By the Hahn-Banach theorem, it suffices to show that for any $F\in (\widetilde{H}^s(\Omega))^*$ such that $F(v)=0$ for any $v\in \mathcal D_2$, we must have $F\equiv 0$. Since $F(v)=0$ for any $v\in \mathcal D_2$, we have 
		\begin{align}\label{Runge 1}
		F(P_{b,c}f-f)=0 \text{ for }f\in C^\infty_c (W).
		\end{align}
		
		Next, we claim that 
		\begin{align}\label{Runge 2}
		F(P_{b,c}f-f)=-B_{b,c}(f,\varphi), \text{ for }f\in C^\infty_c(W),
		\end{align}
		where $\varphi \in \widetilde{H}^s(\Omega)$ is the unique solution of 
		\begin{align*}
		(-\Delta)^s \varphi -\nabla \cdot (b \varphi) +c \varphi =F \text{ in }\Omega \text{ with }\varphi=0 \text{ in }\Omega_e.
		\end{align*}
We remark that by Proposition \ref{prop:well_posed} and the assumption \eqref{eigenvalue condition} this problem is well-posed. In its weak form, it becomes 
		\begin{align*}	
		B_{b,c}^{\ast}(\varphi, w)=F(w) \mbox{ for } w\in \widetilde{H}^s(\Omega).
		\end{align*}
		
		We next address the proof of \eqref{Runge 2}. Let $f\in C^\infty_c(W)$ and $u_f:=P_{b,c}f\in H^s(\R^n)$, then $u_f-f\in \widetilde{H}^s(\Omega)$ and 
		$$
		F(P_{b,c}f-f)=B_{b,c}^{\ast}(\varphi,u_f-f)= B_{b,c}(u_f-f,\varphi)=-B_{b,c}(f,\varphi),
		$$
		where we have utilized \eqref{eq:non_symm} and the fact that $u_f$ is a solution to \eqref{eq:main} and $\varphi \in \widetilde{H}^s(\Omega)$. By means of the relations \eqref{Runge 1} and \eqref{Runge 2}, we hence conclude that 
		$$
		B_{b,c}(f,\varphi )=0 \text{ for }f\in C^\infty_c (W).
		$$
		Since $f =0$ in $\R^n \setminus \overline{W}$ and $\varphi \in \widetilde H^s(\Omega)$, $B_{b,c}(f,\varphi)=0$ implies that 
		$$
		0=((-\Delta)^{s/2}\varphi ,(-\Delta)^{s/2}f)_{\R^n}=((-\Delta)^{s}\varphi,f)_{\R^n},\text{ for }f\in C^\infty_c(W).
		$$
In particular, $\varphi\in H^s(\R^n)$ satisfies 
		$$
		\varphi = (-\Delta)^s \varphi =0 \text{ in }W.
		$$
By invoking Proposition \ref{prop strong unique}, this implies $\varphi \equiv 0 $ in $\R^n$ and therefore $F\equiv 0$ as well.
	\end{proof}
	
The above lemma proves Theorem \ref{thm: Runge approximation} (a). The proof of Theorem \ref{thm: Runge approximation} (b) is postponed to the last section of this paper. Without major modifications of the above argument, we also obtain the Runge approximation for the adjoint equation.
	
	\begin{cor}\label{cor: Runge adjoint}
		For $n\geq 1$ and $\frac{1}{2}<s<1$, let $\Omega \subset \R^n$ be a bounded, open Lipschitz set, $b\in W^{1-s,\infty}(\Omega)^n$ 
		and $c\in L^\infty(\Omega)$ and let $P_{b,c}^{\ast}$ denote the Poisson operator for the equation \eqref{adjoint Dirichlet problem}, i.e. let 
		$$
		P_{b,c}^{\ast}:\mathbb{X} \to H^s(\R^n), \ f^\ast\mapsto  u_{f^\ast }^\ast,
		$$
		where $u_{f^\ast}^\ast \in H^s(\R^n)$ is the solution of the adjoint equation
		$$
		(-\Delta)^su_{f^\ast}^\ast-\nabla \cdot (bu_{f^\ast}^\ast)+cu_{f^\ast}^\ast=0 \text{ in }\Omega \text{ with }u_{f^\ast}^\ast=f^\ast \text{ in }\Omega_e.
		$$
		Let $W\subset \Omega_e$ be an arbitrary open set. Then the sets $\mathcal{D}_1^\ast:=\left\{P_{b,c}^{\ast}f^\ast-f^\ast: \ f^\ast\in C^\infty _c (W)\right\}$ and $\mathcal{D}_2^\ast:=\left\{P_{b,c}^\ast f^\ast-f^\ast: \ f^\ast\in C^\infty _c (W)\right\}$ are dense in $L^2(\Omega)$ and $\widetilde{H}^s(\Omega)$, respectively.	
	\end{cor}

	\section{Global uniqueness}\label{sec:4}
	
	In this section, we prove the global uniqueness result from Theorem \ref{thm main} which follows from the knowledge of the DN map $\Lambda_{b,c}$ and its adjoint.

	\begin{proof}[Proof of Theorem \ref{thm main}]
		Since $\Lambda_{b_1,c_1}f|_{W_2}=\Lambda_{b_2,c_2}f|_{W_2}$ for any $f\in C^\infty_c(W_1)$, where $W_1,W_2$ are arbitrary, but fixed open sets in $\Omega_e$, by using the Alessandrini identity \eqref{eq: Alessandrini}, we have 
		\begin{align}\label{eq: zero integral id}
		\int_\Omega \left((b_1-b_2)\cdot \nabla u_1 u_2^{\ast}+(c_1-c_2)u_1u_2^{\ast}\right) dx=0,
		\end{align}
		where $u_1 \in H^s(\R^n)$ is the solution to $((-\Delta)^s+b_1\cdot \nabla +c_1)u_1=0 $ in $\Omega$ and $u_2^{\ast}\in H^s(\R^n)$ is the solution to $(-\Delta)^s u_2^{\ast}- \nabla \cdot (b_2 u_2^{\ast}) + c_2 u_2^{\ast}=0$ in $\Omega$ with exterior data  $u_1|_{\Omega_e}=f_1\in C^\infty_c(W_1)$ and $u_2^{\ast}|_{\Omega_e}=f_2^{\ast}\in C^\infty_c(W_2)$. In the sequel, we seek to recover the potential and the drift coefficients separately.
		
		\emph{Step 1: Recovery of the potential $c$.} Let $\psi_2 \in C_c^\infty(\Omega)$ be arbitrary. Then choose $\psi_1 \in C_c^\infty(\Omega)$ such that $\psi_1 = 1$ on the set $\supp(\psi_2) \Subset \Omega$. By the Runge approximation of $(-\Delta)^s+b\cdot \nabla +c$ and its adjoint (see Lemma \ref{lem:Runge} and Corollary \ref{cor: Runge adjoint}), there exist sequences of solutions  $\{u_j^1\}_{j=1}^\infty$ and $\{u_j^{2,\ast}\}_{j=1}^\infty$ in $H^s(\R^n)$ such that 
		\begin{align}\label{eq:rungeapply}
		\begin{split}
		&((-\Delta)^s+b_1\cdot \nabla +c_1)u_j^1= (-\Delta)^s u_j^{2,\ast} - \nabla \cdot (b_2 u_j^{2,\ast}) + c_2 u_j^{2,\ast}=0\text{ in }\Omega,\\
		& \mathrm{supp}(u_j^1)\subset \overline{\Omega_1}\text{ and }\mathrm{supp}(u_j^{2,\ast})\subset \overline{\Omega_2},\\
		& u_j^1|_\Omega = \psi_1 +r_j^1 \text{ and }u_j^{2,\ast}|_\Omega = \psi_2+r_j^{2,\ast},
		\end{split}
		\end{align}
		where $\Omega_1, \Omega_2$ are open sets in $\R^n$ containing $\overline{\Omega}$,  
		and $r_j^1, r_j^{2,\ast}\to 0$ strongly in $\widetilde{H}^s(\Omega)$ as $j\to \infty$.

With these solutions at hand, we observe that as $r_j^1, r_j^{2,\ast} \in \widetilde{H}^s(\Omega)$
		\begin{align}
		\label{eq:drift_van}
		\begin{split}
		\left| \int\limits_{\Omega} (b_1-b_2) \cdot \nabla u_j^1 u_j^{2,\ast} dx \right|
		&\leq \left| \int\limits_{\Omega} (b_1 -b_2) \cdot \nabla r_j^{1} r_j^{2,\ast} dx \right| +  \left| \int\limits_{\Omega} (b_1 -b_2) \cdot \nabla r_j^{1} \psi_2 dx \right|\\
		&+ \left| \int\limits_{\Omega} (b_1 -b_2) \cdot \nabla \psi_1 \psi_2 dx \right| + \left| \int\limits_{\Omega} (b_1 -b_2) \cdot \nabla \psi_1 r_j^{2, \ast} dx \right|\\
		& \leq \| \nabla r_j^1\|_{H^{s-1}(\Omega)} \|(b_1-b_2) r_j^{2,\ast}\|_{\widetilde{H}^{1-s}(\Omega)}\\  
		& \quad + \| \nabla r_j^{1}\|_{H^{s-1}(\Omega)} \|(b_1-b_2) \psi_2\|_{\widetilde{H}^{1-s}(\Omega)}\\
		& \quad + \| \nabla \psi_1\|_{H^{s-1}(\Omega)} \|(b_1-b_2) r_j^{2,\ast}\|_{\widetilde{H}^{1-s}(\Omega)}\\
		& \leq  C\|b_1-b_2\|_{W^{1-s,\infty}(\Omega)} \| \nabla r_j^1\|_{H^{s-1}(\Omega)} \|r_j^{2,\ast}\|_{\widetilde{H}^{1-s}(\Omega)}\\
		& \quad + C\|b_1 - b_2\|_{W^{1-s,\infty}(\Omega)}\| \nabla r_j^{1}\|_{H^{s-1}(\Omega)} \|\psi_2\|_{\widetilde{H}^{1-s}(\Omega)}\\
		& \quad + C\|b_1-b_2\|_{W^{1-s,\infty}(\Omega)} \|  r_j^{2, \ast}\|_{\widetilde{H}^{1-s}(\Omega)} \|\nabla \psi_1\|_{H^{s-1}(\Omega)}\\
		& \leq C\|b_1-b_2\|_{W^{1-s,\infty}(\Omega)} \Big(\|r_j^{1}\|_{\widetilde{H}^{s}(\Omega)} \|r_j^{2,\ast}\|_{\widetilde{H}^{s}(\Omega)}\\
		& \quad + \|r_j^{1}\|_{\widetilde{H}^{s}(\Omega)}\|\psi_2\|_{\widetilde{H}^{1-s}(\Omega)} + \| r_j^{2, \ast}\|_{\widetilde{H}^{s}(\Omega)} \|\psi_1\|_{\widetilde{H}^{s}(\Omega)}\Big)\\
		& \rightarrow 0, \text{ as }j \rightarrow \infty.
		\end{split}
		\end{align}	
Here we have used that $b_1-b_2\in W^{1-s,\infty}(\Omega)$ is a bounded multiplier from $\widetilde{H}^{1-s}(\Omega)$ into itself (which follows from the same argument as in the proof of Proposition \ref{prop:well_posed}) and the estimate $\|u\|_{\widetilde{H}^{1-s}(\Omega)} \leq C \|u\|_{\widetilde{H}^{s}(\Omega)}$ for $s\in (\frac{1}{2},1)$. Further, we made strong use of the support assumptions for $\psi_1$ and $\psi_2$, which in particular allow us to drop the third term in the first estimate in \eqref{eq:drift_van}.

		Inserting these solutions $\{u_j^1\}$, $\{u_j^{2,\ast}\}$ and the estimate \eqref{eq:drift_van} into \eqref{eq: zero integral id} and taking $j\to \infty$, together with the assumptions on $\psi_1$ and $\psi_2$, we derive 
		$$
		\int_\Omega (c_1-c_2) \psi_2 dx =0.
		$$
		Since $\psi_2 \in C_c^\infty(\Omega)$ was arbitrary, by density of $C_c^\infty(\Omega)$ in $L^2(\Omega)$ and the previous identity, we obtain that $c_1 = c_2$ in $\Omega$.

		\emph{Step 2: Recovery of the drift $b$.} Since we have $c_1 = c_2 $ in $\Omega$, \eqref{eq: zero integral id} becomes 
		\begin{align}\label{eq: drift identity}
		\int_\Omega (b_1-b_2)\cdot \nabla u_1 u_2^{\ast} dx =0.
		\end{align}
		Fix an arbitrary $\psi_2 \in C_c^\infty(\Omega)$. Then choose $\psi_{x_k}\in C_c^\infty(\Omega)$ equal $x_k$ on $\supp(\psi_2) \Subset \Omega$, where for $k=1,2,\cdots,n$ the function $x_k$ denotes the restriction to the $k$-th component of $x$. 
		By using the Runge approximation again as in \eqref{eq:rungeapply}, we can find sequences of solutions $u_j^1 = \psi_{x_k} + r_j^{1}$ and $u_j^{2, \ast} = \psi_2 + r_j^{2, \ast}$, with $r_j^{1}, r_j^{2, \ast} \to 0$ strongly in $\widetilde{H}^s(\Omega)$ as $j\to \infty$. Plugging the Runge approximations of the functions $\psi_{x_k}$ and $\psi_2$ into \eqref{eq: drift identity}, we obtain 
		$$
		\int_\Omega (b_1-b_2)_k \psi_2 dx + \int\limits_{\Omega} (b_1-b_2) \cdot \nabla r_j^{1} \psi_2 dx +  \int\limits_{\Omega} (b_1-b_2) \cdot \nabla r_j^{1} r_j^{2,\ast} dx 
+ \int\limits_{\Omega} (b_1-b_2) \cdot \nabla{\psi_{x_k}} r_j^{2,\ast} dx=0,
		$$
		for $k=1,2,\cdots,n$, where $(b_1-b_2)_k$ denotes the $k$-th component of the vector valued function $b_1-b_2$. 
		By arguing similarly as in \eqref{eq:drift_van}, in the limit $j\rightarrow \infty$ we arrive at
		\begin{align*}
		\int_\Omega (b_1-b_2)_k \psi_2 dx = 0 \text{ for }k=1,2,\cdots,n.
		\end{align*}		
		Since $\psi_2 \in C_c^\infty(\Omega)$ is arbitrary and as $C_c^\infty(\Omega)$ is dense in $L^2(\Omega)$, we also conclude that $(b_1-b_2)_k=0$ for all $k=1,2,\cdots,n$. Therefore, $b_1=b_2$, which completes the proof. 
	\end{proof}

\section{Stability}\label{sec:stability}

In this section, we study the stability result for the fractional Schr\"odinger equation with drift.

\subsection{Auxiliary results}
\label{sec:aux}
We begin by proving several auxiliary results, which will be used in deducing a quantitative Runge approximation result in the next section. To this end, we will mainly be studying the dual equation to \eqref{eq:main}
\begin{align}
\label{eq:dual}
\begin{split}
(-\D)^s w - \nabla \cdot (b w) + c w & = v \mbox{ in } \Omega,\\
w & = 0 \mbox{ in } \Omega_e.
\end{split}
\end{align}
Throughout this section, we assume that the drift field $b$ and the potential $c$ satisfy \eqref{eigenvalue condition}.

\begin{lem}
	\label{lem:two_sided_est}
	Let $\Omega \subset \R^n$ be an open, bounded Lipschitz domain. Let $s\in(\frac{1}{2},1)$, $b\in W^{1-s,\infty}(\Omega)^n$, $c\in L^{\infty}(\Omega)$, $v\in H^{-s}(\Omega)$ and assume that $w\in H^s_{\overline{\Omega}}$ is the solution to \eqref{eq:dual}.
	Then there exists a constant $C>1$ independent of $v,w$ such that
	\begin{align*}
	C^{-1}\|v\|_{H^{-s}(\Omega)} \leq \|w\|_{H^{s}_{\overline{\Omega}}} \leq C \|v\|_{H^{-s}(\Omega)}.
	\end{align*}
\end{lem}

\begin{proof}
	The upper bound follows from the well-posedness result of Proposition \ref{prop:well_posed} and the assumption \eqref{eigenvalue condition}. It hence remains to discuss the lower bound. To this end, we use the triangle inequality and the equation \eqref{eq:dual}, which lead to

	\begin{align}
	\label{eq:lower_bd}
	\begin{split}
	\|v\|_{H^{-s}(\Omega)}
	& \leq \|(-\D)^s w\|_{H^{-s}(\Omega)} + \|\nabla \cdot (b w)\|_{H^{-s}(\Omega)} + \|c w\|_{H^{-s}(\Omega)}\\
	& \leq \|(-\D)^s w\|_{H^{-s}(\R^n)} + \|\nabla \cdot (b w)\|_{H^{-s}(\Omega)} + \|c w\|_{H^{-s}(\Omega)}\\
	& \leq \| w\|_{H^{s}_{\overline{\Omega}}} + \|\nabla \cdot (b w)\|_{H^{-s}(\Omega)} + \|c w\|_{H^{-s}(\Omega)}.
	\end{split}
	\end{align}
	We estimate the terms with the drift and the potential separately. On the one hand, by integration by parts, we observe that for the drift, we can find a constant $C>0$ independent of $b$ and $w$ such that 
	\begin{align*}
	\|\nabla \cdot (b w)\|_{H^{-s}(\Omega)}
& = \sup\limits_{\|\phi\|_{\widetilde{H}^{s}(\Omega)}=1}\left|(\nabla \cdot (b w), \phi)_{\Omega}\right| \\
& \leq \sup\limits_{\|\phi\|_{\widetilde{H}^{s}(\Omega)}=1} \|b w\|_{\widetilde{H}^{1-s}(\Omega)} \|\nabla \phi \|_{H^{s-1}(\Omega)}\\
	&\leq C \|w\|_{\widetilde{H}^{1-s}(\Omega)} \|b\|_{W^{1-s,\infty}(\Omega)} \sup\limits_{\|\phi\|_{\widetilde{H}^{s}(\Omega)}=1}  \|\phi\|_{\widetilde{H}^{s}(\Omega)}\\
	&\leq C \|w\|_{H^s_{\overline{\Omega}}} \|b\|_{W^{1-s,\infty}(\Omega)} ,
	\end{align*}
	where we used that for $s\in (1/2,1)$ it holds that $\|w \|_{\widetilde{H}^{1-s}(\Omega)} \leq C\|w \|_{\widetilde{H}^{s}_{\overline{\Omega}}}$.
	
	On the other hand, in order to estimate the potential $c$, we use H\"older's and Poincar\'e's inequalities to observe
	\begin{align*}
	\|c w\|_{H^{-s}(\Omega)}
	& = \sup\limits_{\|\phi\|_{H^{s}_{\overline{\Omega}}}=1}\left|(c w, \phi)_{\Omega}\right|
	= \sup\limits_{\|\phi\|_{H^{s}_{\overline{\Omega}}}=1}\|c\|_{L^{\infty}(\Omega)} \|w\|_{L^2(\Omega)} \|\phi\|_{L^2(\Omega)}
	\leq C\|c\|_{L^{\infty}(\Omega)} \|w\|_{H^s_{\overline{\Omega}}}.
	\end{align*}
	Inserting the estimates for the contributions involving $b,c$ into \eqref{eq:lower_bd} then concludes the proof of Lemma \ref{lem:two_sided_est}.
\end{proof}

Next, we prove Vishik-Eskin type estimates for the fractional Schrödinger equation with drift, c.f. \cite{VE65} and also \cite[Section 3]{Gru14}.

\begin{lem}[Vishik-Eskin]
	\label{lem:VE}
	Let $\delta \in (-\frac{1}{2},\frac{1}{2})$. Let $\Omega \subset \R^n$ be an open, bounded Lipschitz domain. Let $s\in(\frac{1}{2},1)$, $b\in W^{1-s+\delta,\infty}(\Omega)^n$, $c\in L^{\infty}(\Omega)$ satisfy \eqref{eigenvalue condition}, $v\in H^{-s}(\Omega)$ and assume that $w\in H^s_{\overline{\Omega}}$ is the solution to \eqref{eq:dual}.
Then there exists a constant $C>1$ such that
	\begin{align*}
	\|w\|_{H^{s+\delta}_{\overline{\Omega}}}
	\leq C \|v\|_{H^{-s+\delta}(\Omega)}.
	\end{align*}
\end{lem}

The argument for this follows from a perturbation of the original estimates due to Vishik and Eskin \cite{VE65}.

\begin{proof}
	We rewrite the equation \eqref{eq:dual} as
	\begin{align*}
	(-\D)^s w & = G \mbox{ in } \Omega,\\
	w&= 0 \mbox{ in } \Omega_e,
	\end{align*}
	where $G= \nabla \cdot (b w) - cw$. Then the estimates of Vishik and Eskin (see \cite[Theorem 3.1]{grubb2015fractional}) yield that
	\begin{align}
	\label{eq:VE}
	\|w\|_{H^{s+\delta}_{\overline{\Omega}}} \leq C \|G\|_{H^{-s+\delta}(\Omega)}
	\leq C (\|cw\|_{H^{-s+\delta}(\Omega)} + \|\nabla \cdot (bw)\|_{H^{-s+\delta}(\Omega)} + \|v\|_{H^{-s+\delta}(\Omega)}).
	\end{align}
	We estimate the drift and the potential contributions separately: For $\delta \in (0,\min\{2s-1,\frac{1}{2}\})$ by integration by parts and duality
	
	\begin{align*}
	\|\nabla \cdot (b w)\|_{H^{-s+\delta}(\Omega)}
	& = \sup\limits_{\|\phi\|_{\widetilde{H}^{s-\delta}(\Omega)}=1}\left|(\nabla \cdot (bw), \phi)_{L^2(\Omega)}\right| \\
	&= \sup\limits_{\|\phi\|_{\widetilde{H}^{s-\delta}(\Omega)}=1}\left|(bw , \nabla \phi)_{L^2(\Omega)}\right|\\
	& \leq \sup\limits_{\|\phi\|_{\widetilde{H}^{s-\delta}(\Omega)}=1} \|b w\|_{\widetilde{H}^{1-s+\delta}(\Omega)} \|\nabla \phi\|_{H^{s-1-\delta}(\Omega)}\\
	& \leq \sup\limits_{\|\phi\|_{\widetilde{H}^{s-\delta}(\Omega)}=1}\|b\|_{W^{1-s+\delta,\infty}(\Omega)} \|w\|_{\widetilde{H}^{1-s+\delta}(\Omega)} \|\phi\|_{\widetilde{H}^{s-\delta}(\Omega)}\\
	&\leq C \|b\|_{W^{1-s+\delta,\infty}(\Omega)} \|w\|_{H^{s}_{\overline{\Omega}}}.
	\end{align*}
	The potential term is estimated by Hölder's and Poincar\'e's inequalities
	\begin{align*}
	\|c w\|_{H^{-s+\delta}(\Omega)}
	& = \sup\limits_{\|\phi\|_{H^{s}_{\overline{\Omega}}}=1}\left|(c w, \phi)_{\Omega}\right| \\
	& = \sup\limits_{\|\phi\|_{H^{s-\delta}_{\overline{\Omega}}}=1}\|c\|_{L^{\infty}(\Omega)} \|w\|_{L^2(\Omega)} \|\phi\|_{L^2(\Omega)}\\
	&\leq C\|c\|_{L^{\infty}(\Omega)} \|w\|_{H^s_{\overline{\Omega}}}.
	\end{align*}
	Combining these bounds with the estimate $\|w\|_{H^{s}_{\overline{\Omega}}} \leq C \|v\|_{H^{-s}(\Omega)}$, which follows from the the well-posedness result of Proposition \ref{prop:well_posed} and returning to \eqref{eq:VE} concludes the argument.
\end{proof}

\subsection{A quantitative approximation result}
\label{sec:quant_approx}

As a final preparation for the stability proof, in this section we deduce a quantitative Runge approximation result for fractional Schrödinger equations with drift terms:

\begin{prop}
	\label{prop:approx}
	For $n\geq 1$, let $\Omega \subset \R^n$ be a bounded open set with a $C^\infty$-smooth boundary and $s\in (\frac{1}{2},1)$. Let  $W \Subset \Omega_e$ be open such that $\overline{\Omega} \cap \overline{W} = \emptyset$. Further suppose that $\delta \in (0,\frac{2s-1}{2})$ and that $b\in W^{1-s+\delta,\infty}(\Omega)^n$, $c \in L^{\infty}(\Omega)$. Then, for each $\epsilon>0$ and for each $\overline{v}\in H^{s}_{\overline{\Omega}}$ there exists $f_{\epsilon} \in H^{s}_{\overline{W}}$ such that the following approximation estimates hold true
	\begin{align*}
	\|P_{b,c} f_{\epsilon} - f_{\epsilon} - \overline{v}\|_{H^{s-\delta}_{\overline{\Omega}}} \leq \epsilon \|\overline{v}\|_{H^{s}_{\overline{\Omega}}}, \ \|f_{\epsilon}\|_{H^{s}_{\overline{W}}} \leq C e^{C \epsilon^{-\mu(\delta)}} \|\overline{v}\|_{H^{s-\delta}_{\overline{\Omega}}}.
	\end{align*}
\end{prop}

As in \cite{RS17} the approximation property follows from a \emph{quantitative} unique continuation result:

\begin{prop}
\label{prop:equiv}
For $n\geq 1$, let $\Omega \subset \R^n$ be a bounded open set with a $C^\infty$-smooth boundary and $s\in (\frac{1}{2},1)$. Let $W \Subset \Omega_e$ be open such that $\overline{\Omega} \cap \overline{W} = \emptyset$. Further suppose that $\delta \in (0,\frac{2s-1}{2})$ and that $b\in W^{1-s+\delta,\infty}(\Omega)^n$, $c \in L^{\infty}(\Omega)$. 

Assume that for each $v\in H^{s-\delta}_{\overline{\Omega}}$ it holds that
\begin{align}
\label{eq:quant_UCP}
\|v\|_{H^{s-2\delta}_{\overline{\Omega}}}
\leq \frac{C}{\left| \log\left( C \frac{\|v\|_{H^{s-\delta}_{\overline{\Omega}}}}{\|(-\D)^s w\|_{H^{-s}(\Omega)}} \right)\right|^{\sigma(\delta)}  } \|v\|_{H^{s-\delta}_{\overline{\Omega}}},
\end{align}
where $w\in H^{s}_{\overline{\Omega}}$ is the solution of \eqref{eq:dual}.
Then, for each $\epsilon>0$ and for each $\overline{v}\in H^{s}_{\overline{\Omega}}$ there exists $f_{\epsilon} \in H^{s}_{\overline{W}}$
such that the following approximation estimate holds true
	\begin{align*}
	\|P_{b,c} f_{\epsilon} - f_{\epsilon} - \overline{v}\|_{H^{s-\delta}_{\overline{\Omega}}} \leq \epsilon \|\overline{v}\|_{H^{s}_{\overline{\Omega}}}, \quad \quad \|f_{\epsilon}\|_{H^{s}_{\overline{W}}} \leq C e^{C \epsilon^{-\mu(\delta)}} \|\overline{v}\|_{H^{s-\delta}_{\overline{\Omega}}}.
	\end{align*}
\end{prop}

This statement is the exact analogue of Lemma 8.2 in \cite{RS17} with the argument for Proposition \ref{prop:equiv} following
verbatim as in the proof of Lemma 8.2 in \cite{RS17}: Indeed, the only property of the equation \eqref{eq:dual} which is used, is its mapping property. By virtue of the regularity assumptions on $b,c$ and the well-posedness results of Proposition \ref{prop:well_posed}, solutions to \eqref{eq:quant_UCP} also enjoy exactly the same regularity and compactness estimates as the ones from \cite{RS17}.

\begin{proof}[Proof of Proposition \ref{prop:equiv}]
We consider the operator
\begin{align*}
A:H^{s}_{\overline{W}} \rightarrow H^{s}_{\overline{\Omega}} \hookrightarrow H^{s-\delta}_{\overline{\Omega}}, \ f \mapsto j(P_{b,c}(f)-f),
\end{align*}
where $j:H^{s}_{\overline{\Omega}} \hookrightarrow H^{s-\delta}_{\overline{\Omega}}$ is a compact embedding.
Thus, $A$ is a compact, injective operator. Here injectivity follows from the strong uniqueness result of Proposition \ref{prop strong unique}. In addition, by (a slight adaptation of) Lemma \ref{lem:Runge}, it has a dense range. Thus, we may apply the spectral theorem for compact operators and obtain sequences $\{\mu_j\}_{j=1}^{\infty}\subset \R_+$ decreasing, and $\{w_j\}_{j=1}^{\infty} \subset H^{s}_{\overline{W}}$ such that $A^{\ast}A w_j = \mu_j w_j$. 

The set $\{w_j\}_{j=1}^{\infty}$ forms an orthonormal basis with respect to the $H^{s}_{\overline{W}}$ scalar product. By the density of the range of $A$, it also follows that the set $\{\varphi_j\}_{j=1}^{\infty}:=\left\{\frac{1}{\sigma_j}A w_j\right\}_{j=1}^{\infty}$ with $\sigma_j := \mu_j^{\frac{1}{2}}$ is an orthonormal basis of $H^{s-\delta}_{\overline{\Omega}}$. As a consequence, $\left\{(\sigma_j, w_j, \varphi_j)\right\}_{j=1}^\infty \subset \R_+\times H^{s}_{\overline{W}}\times H^{s-\delta}_{\overline{\Omega}}$ is the \emph{singular value decomposition} of $A$. By the characterization of $A^{\ast}$, the assumption \eqref{eq:quant_UCP} can be rephrased as the estimate
\begin{align}
\label{eq:quant_UCP_1a}
\|v\|_{H^{s-2\delta}_{\overline{\Omega}}}
\leq \frac{C}{\left| \log\left( C \frac{\|v\|_{H^{s-\delta}_{\overline{\Omega}}}}{\|A^{\ast}v\|_{H^{s}_{\overline{W}}}} \right)\right|^{\sigma(\delta)}} \|v\|_{H^{s-\delta}_{\overline{\Omega}}}.
\end{align}
Using this and the singular value decomposition from above, we deduce the approximation property along the same lines as in \cite[Lemma 8.2]{RS17}: Let $\bar{v}\in H^{s}_{\overline{\Omega}}$. For $\alpha \in (0,1)$, let $r_{\alpha}:= \sum\limits_{\sigma_j\leq \alpha}(\bar{v},w_j)w_j \in H^{s-\delta}_{\overline{\Omega}}$ and $R_{\alpha} \bar{v}:=\sum\limits_{\sigma_j> \alpha}\sigma_j^{-1}(\bar{v},w_j)\varphi_j \in H^{s}_{\overline{W}}$. 

Therefore, on the one hand, we have 
\begin{align}
\label{eq:est_1}
\|R_{\alpha} (\bar{v})\|_{H^{s}_{\overline{W}}} \leq \frac{C}{\alpha}\|\bar{v}\|_{H^{s-\delta}_{\overline{\Omega}}}.
\end{align}
On the other hand,
\begin{align}
\label{eq:est_2}
\begin{split}
&\quad \| \bar{v}-A R_{\alpha}(\bar{v})\|_{H^{s-\delta}_{\overline{\Omega}}}^2 \\
&= \sum\limits_{\sigma_j \leq \alpha}|(\bar{v},w_j)_{H^{s-\delta}_{\overline{\Omega}}}|^2
\leq (\bar{v},r_{\alpha})_{H^{s-\delta}_{\overline{\Omega}}}\\
&=(\bar{v},r_{\alpha})_{H^{s-\delta}(\R^n)}
\leq \|\bar{v}\|_{H^{s}(\R^n)} \|r_{\alpha}\|_{H^{s-2\delta}(\R^n)}
=\|\bar{v}\|_{H^{s}_{\overline{\Omega}}} \|r_{\alpha}\|_{H^{s-2\delta}_{\overline{\Omega}}}\\
&\leq  \|\bar{v}\|_{H^{s}_{\overline{\Omega}}}
 \frac{C}{\left| \log\left( C \frac{\|r_{\alpha}\|_{H^{s-\delta}_{\overline{\Omega}}}}{\|A^{\ast}r_{\alpha}\|_{H^{s}_{\overline{W}}}} \right)\right|^{\sigma(\delta)}} \|r_{\alpha}\|_{H^{s-\delta}_{\overline{\Omega}}}
 \leq  \|\bar{v}\|_{H^{s}_{\overline{\Omega}}}
 \frac{C}{\left|\log\left( C \frac{1}{\alpha} \right)\right|^{\sigma(\delta)}} \|r_{\alpha}\|_{H^{s-\delta}_{\overline{\Omega}}}.
\end{split}
\end{align}
Optimizing \eqref{eq:est_1}, \eqref{eq:est_2} in $\alpha$ then implies the claim.
\end{proof}

As a second main ingredient in the proof of Proposition \ref{prop:approx}, we rely on Theorem 5.1 from \cite{RS17}. This is a \emph{propagation of smallness} estimate from the boundary into the bulk for the Caffarelli-Silvestre extension of a general function. It does not use the specific equation at hand. For completeness, we recall the statement:

\begin{prop}[\cite{RS17}, Proposition 5.1]
\label{prop:51}
For $n\geq 1$, let $\Omega \subset \R^n$ be an open, bounded and smooth domain. Let $W \Subset \Omega_e$ be open, bounded and Lipschitz with $\overline{\Omega}\cap \overline{W} = \emptyset$. Suppose that $s\in (0,1)$ and that $\widetilde{w}$ is a solution to 
\begin{align}
\label{eq:CS_ext}
\begin{split}
\nabla \cdot x_{n+1}^{1-2s}\nabla \widetilde{w} & = 0 \mbox{ in } \R^{n+1}_+,\\
\widetilde{w} & = w \mbox{ on } \R^n \times \{0\},
\end{split}
\end{align}
where $w \in H^{s}(\R^n)$ is a function which vanishes in an open neighbourhood of $W$. Assume further that for some constants $C_1>0$ and $\delta>0$ one has the a priori bounds
\begin{align*}
\left\|x_{n+1}^{1-2s} \p_{n+1} \widetilde{w}\right\|_{H^{-s}(W)}\leq \eta,\\
\left\|x_{n+1}^{\frac{1-2s}{2}} \widetilde{w}\right\|_{L^2(\R^n \times [0,C_1])} + \left\|x_{n+1}^{\frac{1-2s}{2}} \nabla \widetilde{w}\right\|_{L^2(\R^{n+1}_+)}
+ \left\|x_{n+1}^{\frac{1-2s}{2}-\delta} \nabla \widetilde{w}\right\|_{L^2(\R^{n+1}_+)}
 \leq E,
\end{align*}
for some constants $\eta, E$ with $\frac{E}{\eta}>1$. Then, there exist constants $C>1$, $\mu>0$ which depend on $n,s,C_1,\delta, \Omega, W$ such that
\begin{align*}
	\left\|x_{n+1}^{\frac{1-2s}{2}} \widetilde{w}\right\|_{L^2(2\Omega \times [0,1])} + \|x_{n+1}^{\frac{1-2s}{2}} \nabla \widetilde{w}\|_{L^2(2\Omega \times [0,1])}
	\leq C \frac{E}{\left|\log\left( C\frac{E}{\|(-\D)^s w\|_{H^{-s}(W)}} \right)\right| ^{\mu}}.
\end{align*}
Here $2\Omega:= \{x\in \R^n: \ \dist(x, \Omega) \leq \min\{\frac{1}{2} \dist(\overline{W},\overline{\Omega}),2 \}\}$.
\end{prop}

We remark that although Proposition 5.1 in \cite{RS17} is formulated for the Caffarelli-Silvestre extension of the solution to the fractional Schrödinger equation which is studied there (in \cite{RS17} the situation $b=0$ is considered), only the vanishing of $w$ in a neighbourhood of $W$ is used in the argument which leads to Proposition 5.1. We thus do not present the proof of this result, but refer the reader to Section 5 in \cite{RS17}. We will apply it to solutions to \eqref{eq:dual} in the sequel. 

With Propositions \ref{prop:equiv} and \ref{prop:51} at hand, we address the proof of the approximation result of Proposition \ref{prop:approx}:

\begin{proof}[Proof of Proposition \ref{prop:approx}]

By Proposition \ref{prop:equiv}, it suffices to argue that \eqref{eq:quant_UCP} holds true. 

\emph{Step 1: Reduction.}
As in \cite{RS17} the estimate \eqref{eq:quant_UCP} is derived by interpolation from the following quantitative unique continuation result
	\begin{align}
	\label{eq:quant_UCP_1}
	\|v\|_{H^{-s}(\Omega)}
	\leq \frac{C}{ \left| \log\left( \frac{CE }{\|(-\D)^s w\|_{H^{-s}(W)}} \right)\right|^{\sigma}} E,
	\end{align}
	where $E \geq \|v\|_{L^2(\Omega)}$. 
We recall the argument that \eqref{eq:quant_UCP_1} implies \eqref{eq:quant_UCP} (c.f. Step 1 in the proof of Theorem 1.2 in \cite{RS17}): For $\theta \in (0,1)$ such that $0\leq s\theta < \frac{1}{2}$ we first interpolate \eqref{eq:quant_UCP_1} with some $E \geq \|v\|_{L^2(\Omega)}$. This yields 
\begin{align*}
\|v\|_{H^{-\theta s}(\overline{\Omega})}
\leq C^{\theta}\frac{1}{\left| \log\left( \frac{CE }{\|(-\D)^s w\|_{H^{-s}(W)}} \right)\right|^{\sigma(\delta)}} E.
\end{align*}
Using that $H^{-\theta s}(\Omega) = H^{-\theta s}_{\overline{\Omega}}$ for $0<\theta s < \frac{1}{2}$, we then interpolate once more with the trivial bound $\|v\|_{H^{s-\delta}_{\overline{\Omega}}} \leq \|v\|_{H^{s-\delta}_{\overline{\Omega}}}$. Choosing $E \geq \|v\|_{H^{s-\delta}_{\overline{\Omega}}} \geq \|v\|_{L^2(\Omega)}$ then yields
	\begin{align*}
	\|v\|_{H^{-s+2\delta}_{\overline{\Omega}}}
	\leq \frac{C}{ \left|\log\left( \frac{CE }{\|(-\D)^s w\|_{H^{-s}(W)}} \right)\right|^{\widetilde{\sigma}(\delta)} }E.
	\end{align*}
We again emphasize that this reduction does not rely on any properties of the equation \eqref{eq:dual} but only follows from general interpolation arguments.

\emph{Step 2: Proof of \eqref{eq:quant_UCP_1}.}
	Also in the proof of \eqref{eq:quant_UCP_1} only three ingredients involving the solution of \eqref{eq:dual} are exploited:
	\begin{itemize}
		\item[(i)] It is used that $\|w\|_{H^s(\R^n)} \leq C \|v\|_{L^2(\Omega)}$,
		\item[(ii)] that $\|w\|_{H^{s+\widetilde{\delta}}(\R^n)} \leq C \|v\|_{H^{-s+\widetilde{\delta}}(\Omega)} \leq C \|v\|_{L^2(\Omega)}$ for some $\widetilde{\delta}\in (0,1/2)$,
		\item[(iii)] and that $\|v\|_{H^{-s}(\Omega)} \leq C \|w\|_{H^s_{\overline{\Omega}}}$.
	\end{itemize}
	Indeed, choosing a constant $E>0$ such that $\|v\|_{L^2(\Omega)}\leq E$ and assuming that (i), (ii) hold, the properties of the Caffarelli-Silvestre extension yield
	\begin{itemize}
		\item[(i')] $\left\|x_{n+1}^{\frac{1-2s}{2}} \widetilde{w}\right\|_{L^2(\R^n \times [0,C_1])} + \|x_{n+1}^{\frac{1-2s}{2}} \nabla \widetilde{w}\|_{L^2(\R^{n+1}_+)} \leq C E$,
		\item[(ii')] $\left\|x_{n+1}^{\frac{1-2s}{2}-\delta}\widetilde{w}\right\|_{L^2(\R^{n+1}_+)} \leq C E$.
	\end{itemize}
	Here $\widetilde{w}$ denotes the Caffarelli-Silvestre extension of $w$; for the details of the estimates we refer to the proof of Theorem 1.3 in \cite{RS17}. If (i'), (ii') are available and setting $\eta = C\|(-\D)^s w\|_{H^{-s}(W)}$ (so that by the properties of the Caffarelli-Silvestre extension \cite{caffarelli2007extension}, also $\|\lim\limits_{x_{n+1}\rightarrow 0} x_{n+1}^{1-2s}\p_{n+1}\widetilde{w}\|_{H^{-s}(W)} \leq \eta$), Theorem 5.1 in \cite{RS17} is applicable and results in
	\begin{align*}
	\left\|x_{n+1}^{\frac{1-2s}{2}} \widetilde{w}\right\|_{L^2(2\Omega \times [0,1])} + \|x_{n+1}^{\frac{1-2s}{2}} \nabla \widetilde{w}\|_{L^2(2\Omega \times [0,1])}
	\leq C \frac{E}{\left|\log\left( C\frac{E}{\|(-\D)^s w\|_{H^{-s}(W)}} \right)\right|^{\mu}}.
	\end{align*}
	Recalling localized trace estimates (c.f. Lemma 4.4 in \cite{RS17}) then gives
	\begin{align*}
	\|w\|_{H^s_{\overline{\Omega}}}
	\leq C \frac{E}{\left| \log\left( C\frac{E}{\|(-\D)^s w\|_{H^{-s}(W)}} \right) \right|^{\mu}}.
	\end{align*}
	Finally, applying (iii), we can further bound $\|w\|_{H^{s}_{\overline{\Omega}}}$ from below and obtain
	\begin{align*}
	\|v\|_{H^{-s}(\Omega)}
	\leq C \frac{E}{\left| \log\left( C\frac{E}{\|(-\D)^s w\|_{H^{-s}(W)}} \right)\right|^{\mu}}.
	\end{align*}
	Since Lemmas \ref{lem:two_sided_est}, \ref{lem:VE} imply the conditions (i), (ii), (iii), this concludes the proof of Proposition \ref{prop:approx}.
\end{proof}

\subsection{Proof of Theorem \ref{thm:stab}}
\label{sec:stab_proof}

With the approximation result of Proposition \ref{prop:approx} at hand, we present the proof of the stability estimate from Theorem \ref{thm:stab}.

\begin{proof}[Proof of Theorem \ref{thm:stab}]
	As in the uniqueness proof we argue in two steps: 
	
	\emph{Step 1: Stability estimate for $c$.}
	First, we seek to obtain an estimate on  $\|c_1-c_2\|_{H^{-s}(\Omega)}$. Let $\psi_1$ be a smooth function which is compactly supported in $\Omega$ and which satisfies $\psi_1 = 1$ in $\supp(c_1), \supp(c_2), \supp(b_1), \supp(b_2)$. Let further $g\in \widetilde{H}^{s}(\Omega)$ be arbitrary.
Using the quantitative Runge approximation result from Proposition \ref{prop:approx} and considering approximate solutions 
	\begin{align*}
	u^1_j = \psi_1 + r_j^1, \ u^{2,\ast}_j = g + r_j^2, 
	\end{align*}
 with associated exterior data $f_{j}^1, f_{j}^2$, we infer from Alessandrini's identity
	\begin{align}
	\label{eq:stab_1}
	\begin{split}
	((c_1-c_2), g)_{\Omega} 
	& = ((\Lambda_{b_1,c_1}-\Lambda_{b_2,c_2})f_{j}^1,f_{j}^2)_{W_2}
	-((c_1-c_2)r_j^1, g)_{\Omega}  \\
	& \quad - ((c_1-c_2)r_j^1, r_j^2)_{\Omega} - ((b_1-b_2) \cdot \nabla r_{1}^j, u_{j}^{2,\ast})_{\Omega}.
	\end{split}
	\end{align}
	We estimate the terms involving the drift and the potential as follows
	\begin{align}
	\label{eq:stab_2}
	\begin{split}
	|((b_1-b_2)\cdot \nabla r_j^1, u_j^{2,\ast})_{\Omega}|
	&\leq \|\nabla r_j^1\|_{H^{s-1-\delta}(\Omega)} \|(b_1-b_2) u_j^{2,\ast}\|_{\widetilde{H}^{1-s+\delta}(\Omega)}\\
	&\leq C\| r_j^1\|_{H^{s-\delta}(\Omega)} \|b_1-b_2\|_{W^{1-s+\delta,\infty}(\Omega)} 
	\| u_j^{2,\ast}\|_{\widetilde{H}^{1-s+\delta}(\Omega)}\\
	&\leq C M \| r_j^1\|_{H^{s-\delta}(\Omega)} 
	\| u_j^{2,\ast}\|_{\widetilde{H}^{s-\delta}(\Omega)},
	\end{split}
	\end{align}
	where we used that $s\in (\frac{1}{2},1)$ and chose $\delta \in (0,\frac{2s-1}{2})$ sufficiently small, and
	\begin{align}
	\label{eq:stab_3}
	\begin{split}
	|((c_1-c_2) r_j^1, \psi_1)_{\Omega}|
	&\leq \|c_1-c_2\|_{L^{\infty}(\Omega)} \|r_j^1\|_{L^2(\Omega)} \|\psi_1\|_{L^2(\Omega)}\\
	&\leq CM \|r_j^1\|_{\widetilde{H}^{s-\delta}(\Omega)}\|\psi_1\|_{L^2(\Omega)}.
	\end{split}
	\end{align}
	Applying the bounds from Proposition \ref{prop:approx} and combining the bounds from \eqref{eq:stab_2}, \eqref{eq:stab_3} with \eqref{eq:stab_1}, we therefore infer
	\begin{align*}
	|((c_1-c_2), g)_{\Omega}| 
	\leq C(\|\Lambda_{b_1,c_1}-\Lambda_{b_2,c_2}\|_{\ast} e^{C \epsilon^{-\mu}} + \epsilon  M ) \|\psi_1\|_{H^{s}_{\overline{\Omega}}} \|g\|_{H^{s}_{\overline{\Omega}}}.
	\end{align*}
	Optimizing in $\epsilon>0$ by choosing $\epsilon = \left|\log(\|\Lambda_{b_1,c_2}-\Lambda_{b_2,c_2}\|_{\ast})\right|^{-\nu}$ for some $\nu>0$ then implies the estimate 
	\begin{align*}
	\|c_1-c_2\|_{H^{-s}(\Omega)}
	&\leq \sup\limits_{\|g\|_{H^s_{\overline{\Omega}}=1}} \left| ((c_1-c_2),g)_{\Omega}\right|\\
	&\leq \sup\limits_{\|g\|_{H^s_{\overline{\Omega}}=1}} C\left|\log(\|\Lambda_{b_1,c_2}-\Lambda_{b_2,c_2}\|_{\ast})\right|^{-\nu} \|\psi_1\|_{H^s_{\overline{\Omega}}}  \|g\|_{H^s_{\overline{\Omega}}}\\
	&\leq C\left| \log(\|\Lambda_{b_1,c_2}-\Lambda_{b_2,c_2}\|_{\ast})\right|^{-\nu} \|\psi_1\|_{H^s_{\overline{\Omega}}} .
	\end{align*}
	Since $\psi_1$ is fixed, we arrive at the estimate for $\|c_1-c_2\|_{H^{-s}(\Omega)}$.
	
	\emph{Step 2: Stability estimate for $b$.}
	As a preparation for the stability estimate for $b$, we
	next note that by interpolation, for $\widetilde{\delta}>0$ and some $\theta \in (0,1)$, which only depends on $n,s,\widetilde{\delta}$, we have
	\begin{align}
	\label{eq:interpol}
	\begin{split}
	&\quad \|c_1-c_2\|_{L^{\frac{n}{2s}+\widetilde{\delta}}(\Omega)} \\
	&\leq C \|c_1-c_2\|_{L^2(\Omega)}^{\theta} \|c_1-c_2\|_{L^{\infty}(\Omega)}^{1-\theta} \\
	&\leq C \|c_1-c_2\|_{L^2(\Omega)}^{\theta} \|c_1-c_2\|_{W^{1,n+\delta}(\Omega)}^{1-\theta}\\
	&\leq C \|c_1-c_2\|_{H^{-s}(\Omega)}^{\frac{\theta}{2}}  \|c_1-c_2\|_{\widetilde{H}^{s}(\Omega)}^{\frac{\theta}{2}} \|c_1-c_2\|_{W^{1,n+\delta}(\Omega)}^{1-\theta}\\
	&\leq C \left(\left|\log(\|\Lambda_{b_1,c_2}-\Lambda_{b_2,c_2}\|_{\ast})\right|^{-\nu} \|\psi_1\|_{H^s_{\overline{\Omega}}} \right)^{\frac{\theta}{2}}  \|c_1-c_2\|_{\widetilde{H}^{s}(\Omega)}^{\frac{\theta}{2}} \|c_1-c_2\|_{W^{1,n+\delta}(\Omega)}^{1-\theta}.
	\end{split}
	\end{align}
	With this estimate at hand, for $k\in \{1,\dots,n\}$ we again consider approximate solutions
	\begin{align*}
	u_j^1 = \psi_1 x_k + r_j^1, \ u_j^{2,\ast} = g + r_j^2,
	\end{align*}
	with corresponding exterior data $f_{j}^1, f_{j}^{2,\ast}$. Here $\psi_1$ is chosen as above, i.e. in particular such that $\psi_1 = 1$ on the support of the functions $b_{1},b_2, c_1,c_2$. The solutions $u_j^1, u_{j}^{2,\ast}$ are chosen such that the estimates of Proposition \ref{prop:approx} hold.
	
As before, we exploit Alessandrini's identity. Now we have to treat the full term involving $((c_1-c_2) u_{j}^1, u_j^{2,\ast})_{\Omega}$ as an error: For all $k\in\{1,\dots,n\}$
	\begin{align}
	\label{eq:Aless_drift}
	((b_1-b_2)_k, g)_{\Omega} 
	 = &\notag ((\Lambda_{b_1,c_1}-\Lambda_{b_2,c_2})f_{j}^1,f_{j}^2)_{W_2} \\
   &-((c_1-c_2)u_j^1, u_j^{2,\ast})_{\Omega}  
	- ((b_1-b_2) \cdot \nabla r_{1}^j, u_{2}^{j,\ast})_{\Omega}.
	\end{align}
	The terms involving the DN map and the drift term on the right hand side are estimated as in \eqref{eq:stab_2}, \eqref{eq:stab_3}. For the term involving the difference of the potentials, we invoke the interpolation estimate from above; we estimate it as follows
	\begin{align*}
	&|((c_1-c_2) u_{j}^1, u_j^{2,\ast})_{\Omega}| \\
	\leq & \|c_1-c_2\|_{L^{\frac{n}{2s}+\widetilde{\delta}}(\Omega)} \|u_{j}^1\|_{L^{2^{\star}-\widetilde{\delta}}(\Omega)} \|u_{j}^{2,\ast}\|_{L^{2^{\star}-\widetilde{\delta}}(\Omega)}\\
	\leq & C\|c_1-c_2\|_{H^{-s}(\Omega)}^{\frac{\theta}{2}}  \|c_1-c_2\|_{\widetilde{H}^{s}(\Omega)}^{\frac{\theta}{2}} \|c_1-c_2\|_{W^{1,n+\delta}(\Omega)}^{1-\theta} \|u_{j}^1\|_{\widetilde{H}^{s-\delta}(\Omega)} \|u_{j}^{2,\ast}\|_{\widetilde{H}^{s-\delta}(\Omega)}\\
	\leq & C \left|\log(\|\Lambda_{b_1,c_1}-\Lambda_{b_2,c_2}\|_{\ast})\right|^{-\nu \frac{\theta}{2}}  \|c_1-c_2\|_{\widetilde{H}^{s}(\Omega)}^{\frac{\theta}{2}} \|c_1-c_2\|_{W^{1,n+\delta}(\Omega)}^{1-\theta} \|u_{j}^1\|_{\widetilde{H}^{s-\delta}(\Omega)} \|u_{j}^{2,\ast}\|_{\widetilde{H}^{s-\delta}(\Omega)}.
	\end{align*}
	Here $2^{\star}= \frac{2n}{n-2s}$ denotes the corresponding (fractional) Sobolev embedding exponent and $\widetilde{\delta}>0$ is sufficiently small.
	Inserting this into \eqref{eq:Aless_drift} and using the bounds from Proposition \ref{prop:approx}, we infer
	\begin{align*}
	&|((b_1-b_2)_k, g)_{\Omega}|
	\\
	\leq & C (1+M) \left( \|\Lambda_{b_1,c_1}-\Lambda_{b_2,c_2}\|_{\ast} e^{C \epsilon^{-\mu}} + \left|\log(\|\Lambda_{b_1,c_1}-\Lambda_{b_2,c_2}\|_{\ast})\right|^{-\nu \frac{\theta}{2}} + \epsilon \right)
	\|\psi_1 x_k\|_{H^s_{\overline{\Omega}}} \|g\|_{H^s_{\overline{\Omega}}}.
	\end{align*}
	for any $k\in\{1,\dots,n\}$.
	Optimizing once more in $\epsilon$ and taking the supremum over $g\in \widetilde{H}^{s}(\Omega)$ with $\|g\|_{\widetilde{H}^s(\Omega)}=1$ and over $k\in\{1,\dots,n\}$, hence also leads to the desired logarithmic stability estimate for the difference of the drifts.
	This concludes the proof of the stability estimate.
\end{proof}

\section{Reconstruction and finite measurements uniqueness}\label{sec:rmk_reconst}

Last but not least, we present a few results on reconstruction procedures and finite measurement statements for the fractional Calder\'on problem with drift. More precisely, we show that uniqueness for the fractional Calder\'on problem with $C^\infty_c$ drift and potential can be guaranteed from $n+1$ measurements only. 

\subsection{Higher order approximation}
Before addressing the finite measurement results, we recall the higher order Runge approximation property and present the proof of Theorem \ref{thm: Runge approximation} (b).
The structure of the proof for the higher order Runge approximation is similar as the arguments presented in Section \ref{sec:3}. However, since we seek to approximate solutions in high regularity function spaces, by using a duality argument, we need to consider the corresponding Dirichlet problem in Sobolev spaces of negative orders.
The argument for this follows along the same lines as the proofs in \cite[Section 7]{ghosh2016calder}, which in the sequel we recall for self-containedness for the fractional Schr\"odinger equation with drift. 

Let us consider the fractional Laplacian $(-\Delta)^s$ with $s \in (\frac{1}{2},1)$. Let $\Omega\subset \R^n$ be a bounded domain with a $C^\infty$-smooth boundary, $b\in C^\infty_c(\Omega)^n$, $c\in C^\infty_c(\Omega)$ with $\supp(b),\supp(c)\Subset \Omega$ satisfy \eqref{eigenvalue condition}. Here we impose the described compact support as well as the high regularity conditions for $b$ and $c$ in order to satisfy the assumptions from the theory which is presented in \cite{grubb2015fractional}. In the sequel, we consider the function space
$$
\mathcal{E}^s(\overline{\Omega}):=e_\Omega d(x)^s C^\infty(\overline{\Omega}),
$$
where $e_\Omega$ denotes extension by zero from $\Omega$ to $\R^n$, and $d=d(x)$ is a $C^\infty$ function in $\overline{\Omega}$ with $d(x)>0$ for $x\in \Omega$ and $d(x)=\dist(x,\partial \Omega)$ near $\partial \Omega$.  

Further for $\mu> s- \frac{1}{2}$ we work in the Banach space $H^{s(\mu)}(\overline{\Omega})$ which is the space which is introduced in \cite{grubb2015fractional} as the Banach space tailored for solutions $u$ solving the problem 
$$
r_\Omega ((-\Delta)^s + b\cdot \nabla +c)u \in H^{\mu-2s}(\Omega) \text{ with }u=0 \text{ in }\Omega_e,
$$
where $r_\Omega$ is the restriction map from $\R^n$ to $\Omega$ such that $r_\Omega u=u|_\Omega$.

In order to deduce the desired higher order approximation property, we recall the following result from \cite{grubb2015fractional}, which was also used in \cite[Lemma 7.1]{ghosh2016calder} for deducing higher order approximation for the fractional Schr\"odinger equation.

\begin{prop}[Lemma 7.1 in \cite{ghosh2016calder}]
\label{prop:higher order tools}
	For $\mu>s-\frac{1}{2}$ and a smooth bounded domain $\Omega \subset \R^n$, there exists a Banach space $H^{s(\mu)}(\overline{\Omega})$ with the following properties:
	\begin{itemize}
		\item[(a)] $H^{s(\mu)}(\overline{\Omega})\subset H^{s-\frac{1}{2}}_{\overline{\Omega}}$ with a continuous inclusion;
		
		\item[(b)] $H^{s(\mu)}(\overline{\Omega})=H^{\mu}_{\overline{\Omega}}$ for $\mu\in (s-\frac{1}{2},s+\frac{1}{2})$;
		
		\item[(c)] $r_\Omega((-\Delta)^{s}+b\cdot \nabla +c)$ is a homeomorphism from $H^{s(\mu)}(\overline{\Omega})$ onto $H^{\mu-2s}(\Omega)$;
		
		\item[(d)] $H^{\mu}_{\overline{\Omega}}\subset H^{s(\mu)}(\overline{\Omega})\subset H^{\mu}_{loc}(\Omega)$ with continuous inclusions, or the multiplication by any smooth cut-off $\chi\in C^\infty_c(\Omega)$ is bounded from $H^{s(\mu)}(\overline{\Omega})$ to $H^{\mu}(\Omega)$;
		
		\item[(e)] $\mathcal{E}^s (\overline{\Omega})=\cap_{\mu>s-\frac{1}{2}}H^{s(\mu)}(\overline{\Omega})$ and the set $\mathcal{E}^s (\overline{\Omega})$ is dense in $H^{s(\mu)}(\overline{\Omega})$.
	\end{itemize}
\end{prop}

For the proof of this result we refer to \cite{Gru14,grubb2015fractional} and \cite{ghosh2016calder}. We remark that equipped with the topology induced by $\{\|\cdot\|_{H^{s(k)}}\}_{k=1}^\infty$, the space $\mathcal{E}^s (\overline{\Omega})$ is a Fr\'echet space.

Building on these properties of the spaces $H^{s(\mu)}(\overline{\Omega})$, following the argument of \cite{ghosh2016calder}, we prove a higher order approximation property for solutions to \eqref{eq:main} in $\mathcal{E}^s (\overline{\Omega})$. The following result was proved by \cite[Lemma 7.2 ]{ghosh2016calder} for the case $b=0$.

\begin{lem}
\label{lem:Higher Runge}
	Let $\Omega \subset \R^n$ be a bounded domain with a $C^\infty$-smooth boundary, and $\frac{1}{2}<s <1$. Let $W\subset \Omega_e$ be an open set, and let $b\in C^\infty_c(\Omega)^n$, $c\in C^\infty_c(\Omega)$ with $\supp(b),\supp(c)\Subset \Omega$ be such that \eqref{eigenvalue condition} holds. Let $P_{b,c}$ be the Poisson operator given by \eqref{Poisson operators} and 
	$$
	\mathcal{D}:=\{e_\Omega(r_\Omega P_{b,c}f): f\in C^\infty_c(W)\}.
	$$
	Then the set $\mathcal{D}$ is dense in the Fr\'echet space $\mathcal{E}^s (\overline{\Omega})$ with the topology induced by $\{\|\cdot\|_{H^{s(k)}}\}_{k=1}^\infty$.
\end{lem}

\begin{proof}
We follow the argument of \cite[Lemma 7.2]{ghosh2016calder}.
	First, notice that for any $f\in C^\infty_c(W)$, by the definition of the Poisson operator \eqref{Poisson operators}, one has $P_{b,c}f=f+v$, where $v \in \widetilde H^s(\Omega)$ satisfies 
$$
r_\Omega((-\Delta)^s+b\cdot \nabla +c)v\in C^\infty(\overline{\Omega}).
$$
By Proposition \ref{prop:higher order tools}, we have $v\in \mathcal{E}^s(\overline{\Omega})$, which implies that $\mathcal{D}\subset \mathcal{E}^s (\overline{\Omega})$. 
Next, let $\mathcal{L}$ be a continuous linear functional defined on $\mathcal{E}^s (\overline{\Omega})$ satisfying
	$$ 
	\mathcal{L}(e_\Omega (r_\Omega P_{b,c}f))=0, \text{ for all }f\in C^\infty_c(W).		
	$$
	By the Hahn-Banach theorem, it suffices to show that $\mathcal{L}\equiv0$. By using the definition of the topology of the Fr\'echet space $\mathcal{E}^s (\overline{\Omega})$, one can find an integer $\ell$ so that 
	$$
	|\mathcal{L}(u)|\leq C\sum_{m=1}^{\ell} \|u\|_{H^{s(m)}(\overline{\Omega})}\leq C\|u\|_{H^{s(\ell)}(\overline{\Omega})}, \text{ for }u\in \mathcal{E}^{s}(\overline{\Omega}),
	$$
for some constant $C>0$ which is independent of $u$.
By virtue of Proposition \ref{prop:higher order tools} (e), $\mathcal{E}^{s}(\overline{\Omega})$ is dense in $H^{s(\ell)}(\overline{\Omega})$. Thus, $\mathcal{L}$ has a unique bounded extension $\widetilde{\mathcal{L}}\in (H^{s(\ell)}(\overline{\Omega}))^{\ast}$. 

Let us consider the same homeomorphism in Proposition \ref{prop:higher order tools} (c), 
	$$
	\mathcal{T}=r_\Omega((-\Delta)^s+b\cdot \nabla +c):H^{s(\ell)}(\overline{\Omega})\to H^{\ell-2s}(\Omega).
	$$
	The adjoint of $\mathcal T$ is a bounded map between the dual Banach spaces with 
	$$
	\mathcal{T}^{\ast}:(H^{\ell-2s}(\Omega))^{\ast}\to (H^{s(\ell)}(\overline{\Omega}))^*.
	$$
	Note that the adjoint map $\mathcal{T}^\ast$ is also homeomorphism with the inverse $(\mathcal{T}^{-1})^{\ast}$. Moreover, by Remark \ref{rmk on Sobolev}, we have $(H^{\ell-2s}(\Omega))^\ast =H^{-\ell+2s}_{\overline{\Omega}}$ such that 
	$$
	\mathcal{T}^{\ast}v(w)=(v,\mathcal{T}w)_{H^{-\ell+2s}_{\overline{\Omega}} \times H^{\ell-2s}(\Omega)}, \text{ for }v\in H^{-\ell+2s}_{\overline{\Omega}} \text{ and }w\in H^{s(\ell)}(\overline{\Omega}).
	$$
	
	Let $v\in H^{-\ell+2s}_{\overline{\Omega}}$ be the unique function satisfying $\mathcal{T}^{\ast}v=\widetilde{\mathcal{L}}$ and choose a sequence $\{v_k\}_{k\in \mathbb{N}}\subset C^\infty_c(\Omega)$ with $v_k \to v$ in $H^{-\ell+2s}(\Omega)$ as $k\to \infty$. Now, let $f\in C^\infty_c(W)$, recalling that $e_\Omega (r_\Omega P_{b,c}f)=P_{b,c}f-f$, then we have 
	\begin{align}\label{high runge calculation 1}
	0=&\notag\mathcal{L}(e_\Omega(r_\Omega P_{b,c}f))=\widetilde{\mathcal{L}}(P_{b,c}f-f)=\mathcal{T}^\ast v(P_{b,c}f-f)\\ 
	= &\notag (v,\mathcal{T}(P_{b,c}f-f))	= -(v,\mathcal{T}f)=\lim_{k\to \infty}(v_k,((-\Delta)^s +b\cdot \nabla +c)f)\\
	=&-\lim_{k\to \infty}((-\Delta)^s v_k -\nabla \cdot (bv_k)+cv_k, f),
	\end{align}
	where we have utilized that $\mathcal{T}P_{b,c}f=0$ and $v_k\in C^\infty_c(\Omega)$. Finally, since $f\in C^\infty_c(W)$ with $\overline W\cap \overline\Omega=\emptyset$, then the last equation of \eqref{high runge calculation 1} reads
	$$
	((-\Delta)^s v,f)=\lim_{k\to\infty}((-\Delta)^s v_k,f)=0, \text{ for }f\in C^\infty_c(W).
	$$
	Thus, we obtain that $v\in H^{-\ell+2s}(\R^n)$ satisfies 
	$$
	v|_W=(-\Delta)^s v|_W=0,
	$$
	and the strong uniqueness (Proposition \ref{prop strong unique}) implies that $v\equiv 0 $ in $\R^n$. Therefore, we obtain $\widetilde{\mathcal{L}}=0$ and hence $\mathcal{L}=0$, which completes the proof.
\end{proof}

Now, we are ready to prove the higher regularity Runge approximation property.

\begin{proof}[Proof of Theorem \ref{thm: Runge approximation} (b)]
As $\Omega \Subset \Omega_1$ with $\mathrm{int}(\Omega_1 \setminus \overline{\Omega})\neq \emptyset$, it is possible to find a small ball $W$ with $\overline{W}\subset \Omega_1\setminus \overline{\Omega}$. Let $g\in C^\infty(\overline{\Omega})$ and 
	$h:=e_{\Omega} d(x)^{s} g \in \mathcal{E}^{s}(\overline{\Omega})$, then Lemma \ref{lem:Higher Runge} shows that one can find a sequence of solutions $\{u_j\} \subset H^s(\R^n)$ satisfying 
	$$
	((-\Delta)^s+b\cdot \nabla +c)u_j=0 \text{ in }\Omega \text{ with }\supp(u_j)\subset \overline{\Omega_1},
	$$
	so that $e_\Omega r_\Omega u_j \in \mathcal{E}^s(\overline{\Omega})$ and 
	$$
	e_\Omega r_\Omega u_j \to h \text{ in }\mathcal{E}^{s}(\overline{\Omega}) \text{ as }j\to \infty.
	$$
	The higher order approximation will hold if we can show that 
	$$
	\mathcal{M}:C^\infty(\Omega) \to \mathcal{E}^{s}(\overline{\Omega}) \text{ with }\mathcal{M}g=e_\Omega d(x)^s g
	$$
	is a homeomorphism, as it is then possible to apply $\mathcal{M}^{-1}=d(x)^{-s}r_\Omega$. This then gives  
	$$
	d(x)^{-s}r_\Omega u_j \to g \text{ in }C^\infty(\overline{\Omega}).
	$$
	Note that the map $\mathcal{M}$ is a bijective linear map between Fr\'echet spaces and has a closed graph, i.e., if $g_j \to g$ in $C^\infty$ and $\mathcal{M}g_j\to h$ in $\mathcal{E}^s$, then also $\mathcal{M}g_j \to \mathcal{M}g$ in $L^\infty$. Then by the uniqueness of the limit, one obtains that $\mathcal{M}g=h$ as distributional limits. Hence, $\mathcal{M}$ is a homeomorphism by the closed graph and the open mapping theorems. This finishes the proof.
\end{proof}

\subsection{Finite measurements reconstruction without openness}
\label{subsec:Finite measurements reconstruction}
In this section, we discuss a first result towards the proof of Theorem \ref{prop:finite_meas} by using higher order Runge approximation (Theorem \ref{thm: Runge approximation} (b)). However, before proving the full result of Theorem \ref{prop:finite_meas}, we prove a weaker (but technically considerably easier) result, which still proves finite measurement reconstruction but only asserts that the set of measurement data contains a non-empty open set (a priori this argument does not prove the \emph{density} of the set of good data). The technically more involved statement on the openness and density of the set of good data will be proved in the subsequent sections.

\begin{proof}[Proof of Theorem \ref{prop:finite_meas} without the density result]
	We show that for any drift $b \in C^{\infty}_c(\Omega)^n$ and any potential $c\in C^{\infty}_c(\Omega)$ with $\supp(b),\supp(c)\Subset \Omega$, there exist exterior Dirichlet data $f_1,\dots,f_{n+1}$ such that the $b$ and $c$ can be uniquely reconstructed from the knowledge of $f_1,\dots,f_{n+1}$ and $\Lambda_{b,c}(f_1),\dots, \Lambda_{b,c}(f_{n+1})$.
	
	By Runge approximation in $C^{k}$ spaces (see Theorem \ref{thm: Runge approximation} (b)), we have that for any $g\in C^{\infty}(\overline{\Omega})$ there exists a sequence of solutions $\{u_j\}_{j \in \N}$ to the fractional Schrödinger equation \eqref{eq:main} with drift (and compactly supported coefficients) such that for any $k \in \N$
	\begin{align*}
	\|g - d^{-s} u_j \|_{C^k(\Omega)} \rightarrow 0 \text{ as }j\to \infty,
	\end{align*}
	where $d(x)= \dist(x,\partial \Omega)$ if $x\in \Omega$ is sufficiently close to the boundary of $\Omega$ and $d(x)$ is extended to a positive function smoothly into the interior of $\Omega$. Next, we choose $n+1$ smooth functions $g_1,\dots,g_{n+1}$ defined in $\Omega$ with the property that
	\begin{align}
	\label{eq:h}
	h(g_1,g_2,\hdots,g_{n+1})(x) := \det
	\begin{pmatrix}
	\p_1 g_1  & \hdots & \p_n g_1 & g_1 \\
	\vdots & \vdots & \vdots & \vdots\\
	\p_1 g_{n} & \hdots & \p_n g_{n} & g_{n}\\
	\p_1 g_{n+1} & \hdots & \p_n g_{n+1} & g_{n+1}
	\end{pmatrix} (x)\neq 0.
	\end{align}
	An example for this would be the functions $g_j= x_j$ for $j\in\{1,\dots,n\}$ and $g_{n+1}=1$.
	We further set $\widetilde{g}_l= d^{-s} (\chi g_l)$, where $\chi \in C_c^\infty(\Omega)$ and $\chi = 1$ on $K$, for $l\in\{1,\dots,n+1\}$ and apply Theorem \ref{thm: Runge approximation} (b). As a consequence, for each $l\in \{1,\dots,n+1\}$, and in any compact subset $K \Subset \Omega$ there exists a sequence of solutions $\{u_{j,K}^l\}_{j\in \N}$ such that for any $k\in \N$
	\begin{align*}
	\|d^{-s}( g_l-u_{j,K}^l)\|_{C^k(K)} \rightarrow 0 \mbox{ as } j \rightarrow \infty.
	\end{align*}
As $K \subset \Omega$ and as $d(x)>0$ in $K$, we then also have
	\begin{align*}
	\| g_l-u_{j,K}^l\|_{C^k(K)} \rightarrow 0 \mbox{ as } j \rightarrow \infty.
	\end{align*}
	
	Hence, choosing $j\geq j_0$ large enough and $K$ such that $\supp(b)\cup \supp(c)\Subset K$, we obtain that 
	\begin{align}\label{eq:non-zero}
	h( u_{j,K}^1, \hdots, u_{j,K}^{n+1})(x)= \det
	\begin{pmatrix}
	\p_1 u_{j,K}^1  & \hdots & \p_n u_{j,K}^1 & u_{j,K}^1 \\
	\vdots & \vdots & \vdots & \vdots\\
	\p_1 u_{j,K}^n & \hdots & \p_n u_{j,K}^n & u_{j,K}^n\\
	\p_1 u_{j,K}^{n+1} & \hdots & \p_n u_{j,K}^{n+1} & u_{j,K}^{n+1}
	\end{pmatrix} (x)\neq 0.
	\end{align}
	As a consequence, for these values of $u_{j,K}^l$ the (linear) system for $b_1,\dots,b_n$ and $c$
	\begin{align}\label{eq:reconstruction}
	\begin{pmatrix}
	\p_1 u_{j,K}^1  & \hdots & \p_n u_{j,K}^1 & u_{j,K}^1 \\
	\vdots & \vdots & \vdots & \vdots\\
	\p_1 u_{j,K}^n & \hdots & \p_n u_{j,K}^n & u_{j,K}^n\\
	\p_1 u_{j,K}^{n+1} & \hdots & \p_n u_{j,K}^{n+1} & u_{j,K}^{n+1}
	\end{pmatrix}
	\begin{pmatrix}
	b_1 \\
	\vdots \\
	b_n\\
	c 
	\end{pmatrix}
	= -\begin{pmatrix}
	(-\D)^s u_{j,K}^1\\
	\vdots \\
	(-\D)^s u_{j,K}^n \\
	(-\D)^s u_{j,K}^{n+1}
	\end{pmatrix}
	\end{align}
	is solvable (since the $(n+1)\times (n+1)$ matrix in the left hand side of \eqref{eq:reconstruction} is invertible). Thus, from the knowledge of $u_{j,K}^l$, $l\in\{1,\dots,n+1\}$ for some $j\geq j_0$, it is possible to uniquely recover the drift and the potential simultaneously.
	
	As by the global (nonlocal) unique continuation arguments in \cite{GRSU18} (c.f. also Proposition \ref{prop strong unique} from above) it is possible to recover $u_{j,K}^{l}$ given the measurements $f_{j,K}^l$ and $\Lambda_{b,c}(f_{j,K}^l)$ we infer the finite measurement recovery statement.
	
	Finally, in order to infer the openness of the set of possible exterior data, we note that for any $\epsilon>0$ there exists $\delta>0$ such that for any $f=(f^1,\dots, f^{n+1})\in C_c^{\infty}(W)^{n+1}$ with
	\begin{align*}
	\|f-f_{j,K}\|_{C^{k}_{c}(W)} < \delta , \ k \in \N,
	\end{align*}
	we have by boundedness of the mapping $C_c^k(W)^{n+1} \ni f \mapsto u \in C^{k}(K)$ 
	\begin{align*}
	\|u-u_{j,K}\|_{C^k(K)} < \epsilon.
	\end{align*}
	Here $f_{j,K}:=(f_{j,K}^1,\dots,f_{j,K}^{n+1})\in C_c^{\infty}(W)^{n+1}$ are the exterior data from above, $u=(u^1,\dots,u^{n+1})$ are the solutions to \eqref{eq:main} corresponding to the data $f$ and $u_{j,K}:=(u_{j,K}^1,\dots,u_{j,K}^{n+1})$ are the solutions to \eqref{eq:main} corresponding to the data $f_{j,K}=(f_{j,K}^1,\dots,f_{j,K}^{n+1})$. 
	In particular, assuming that 
	\begin{align*}
	\det
	\begin{pmatrix}
	\p_1 u_{j,K}^1  & \hdots & \p_n u_{j,K}^1 & u_{j,K}^1 \\
	\vdots & \vdots & \vdots & \vdots\\
	\p_1 u_{j,K}^n & \hdots & \p_n u_{j,K}^n & u_{j,K}^n\\
	\p_1 u_{j,K}^{n+1} & \hdots & \p_n u_{j,K}^{n+1} & u_{j,K}^{n+1}
	\end{pmatrix} \geq c>0 \mbox{ in } K,
	\end{align*}
	the triangle inequality implies that if $\epsilon>0$ (and correspondingly $\delta$) is chosen sufficiently small, also 
	\begin{align*}
	\det
	\begin{pmatrix}
	\p_1 u^1  & \hdots & \p_n u^1 & u^1 \\
	\vdots & \vdots & \vdots & \vdots\\
	\p_1 u^n & \hdots & \p_n u^n & u^n\\
	\p_1 u^{n+1} & \hdots & \p_n u^{n+1} & u^{n+1}
	\end{pmatrix} >0 \mbox{ in } K.
	\end{align*}
	This concludes the argument.
\end{proof}

\begin{rmk}
	\label{rmk:support}
	We conclude this section by some comments on the assumptions of the theorem:
	\begin{itemize}
		\item[(a)] The compact support condition for the functions $b_j$, $c_j$, $j=1,2$, is assumed here in order to be able to apply the theory of Grubb \cite{grubb2015fractional}.
		
		\item[(b)] In order to obtain our result we in principle do not need the full strength of the $C^k$, $k\in \N$, approximation result from \cite{ghosh2016calder}. It would for instance be sufficient to use a $C^1$ approximation result for which only lower regularity on the coefficients is needed. Since the theory of Grubb is however formulated in the smooth set-up, we do not optimize the regularity dependences here.
		
		\item[(c)] We again emphasize that although the variant of Theorem \ref{prop:finite_meas} which is proved in this section is interesting from a theoretical point of view, a word of caution is needed as follows: In contrast to the single measurement result from \cite{GRSU18}, the exterior data $f_1,\dots,f_{n+1}$ are \emph{not} arbitrary. In general the specific choice of these data depends on $b_1, b_2$, $c_1,c_2$ and hence the explicit choice of the functions $f_1,\dots, f_{n+1}$ is \emph{not} known in general. This has a similar background (see the following proof) as many results on \emph{hybrid inverse problems}, where it is important that constraints are satisfied, see \cite{BU13, A15}. The result from this section will be improved considerably in the argument leading to the proof of Theorem \ref{prop:finite_meas}.
		
		\item[(d)] In addition to the previous point, there are examples of matrices with entries satisfying elliptic equations, for which the determinant vanishes on an open set, see \cite{cekicdeterminants18}. This indicates that the zero set of the determinant \eqref{eq:non-zero} can indeed be \emph{large}.
	\end{itemize}
\end{rmk}

\subsection{Variations and extensions of the reconstruction result}
We discuss a slight variation of the reconstruction result from Section \ref{subsec:Finite measurements reconstruction} by relaxing the condition that the fields $b$, $c$ are compactly supported in $\Omega$. Recall that a $C^\infty$-smooth function $f$ is \emph{vanishing to infinite order} at a point $x_0$ provided that $\p^\alpha f(x_0)=0$ holds for any multi-index $\alpha=(\alpha_1,\cdots,\alpha_n) \in (\N \cup \{0\})^n$.

\begin{prop}
\label{prop:reconstr_a}
Let $\Omega \subset \R^n$ be a bounded domain with a $C^\infty$-smooth boundary. Let $W \subset \Omega_e$ be an open, smooth domain containing an open neighbourhood of $\partial \Omega$. Let $\frac{1}{2}<s<1$ and assume that $b_j\in C^{\infty}(\overline\Omega)^n$, $c_j\in C^{\infty}(\overline\Omega)$ satisfy \eqref{eigenvalue condition} and  
\begin{align*}
b_1-b_2, \ c_1-c_2 \mbox{ vanish to infinite order on } \partial \Omega.
\end{align*} 
Then, there exist $n+1$ exterior Dirichlet data $f_1,\dots,f_{n+1}\in C^{\infty}_c(W)$ such that if
		\begin{align*}
		\Lambda_{b_1,c_1}(f_l) = \Lambda_{b_2,c_2}(f_l) \mbox{ for } l \in\{1,\dots,n+1\},
		\end{align*}
		then $b_1 = b_2$ and $c_1=c_2$ in $\Omega$.
		Moreover, for any $k\in \N$ the set of exterior data $f_1,\dots,f_{n+1}$, which satisfies this property forms an open subset in $C^{k}_{ c}(W)$.
\end{prop}

This follows from an auxiliary result, which states that under geometric restrictions on the set where we measure the Dirichlet data, we may enlarge our domain and that the DN map on the larger domain is determined by the DN map on the smaller set. This is well-known in the study of local inverse problems -- see e.g. \cite[Lemma 4.2]{salo2004inverse}.

In what follows, we denote by $r_X \widetilde{H}^s(Y)$ the set of all restrictions $f|_{X}$ of functions $f \in \widetilde{H}^s(Y)$ for open $X, Y \subset \mathbb{R}^n$.

\begin{lem}
\label{lem:ext}
	Assume that $\Omega$ and $\Omega'$ with $\Omega \subset \Omega'$ are bounded domains with Lipschitz boundaries and $W_1, W_2 \subset \mathbb{R}^n$ are two open sets, such that $\Omega'\setminus \Omega \Subset W_1$ (so, in particular, $W_1 \cap \Omega'_e$ is non-empty). Then let $b_1, b_2 \in W^{1-s, \infty}(\Omega')$ and $c_1, c_2 \in L^\infty(\Omega')$, and $\frac{1}{2} < s < 1$, and assume $c_1 = c_2$ and $b_1 = b_2$ in $\Omega' \setminus \Omega$. 
	
	Additionally, assume that the coefficients satisfy the eigenvalue condition \eqref{eigenvalue condition} both on $\Omega'$ and on $\Omega$. Assume the equality of DN maps $\Lambda_{b_j,c_j}$ with respect to $\Omega$ 
	$$
	\Lambda_{b_1, c_1}f|_{W_2} = \Lambda_{b_2, c_2}f|_{W_2} \text{ for all } f \in r_{W_1 \cap \Omega_e}\widetilde{H}^s(W_1),
	$$ 
	Then we have the equality of DN maps $\Lambda'_{b_j,c_j}$ with respect to $\Omega'$: 
	$$\Lambda'_{b_1, c_1}f|_{W_2} = \Lambda'_{b_2, c_2}f|_{W_2} \text{ for all }f \in r_{W_1 \cap \Omega'_e}\widetilde{H}^s(W_1).
	$$
\end{lem}

\begin{proof}
	Assume $u_1' \in H^s(\mathbb{R}^n)$ solves the Dirichlet problem for $f \in \widetilde{H}^s(W_1 \cap \Omega'_e)$
	\begin{align*}
		((-\Delta)^s + b_1 \cdot \nabla +  c_1)u_1' &= 0 \, \text{ in } \, \Omega',\\
		u_1'|_{\Omega'_e} &= f|_{\Omega'_e}.
	\end{align*}
	Then solve the analogous Dirichlet problem with respect to the smaller domain $\Omega$:
	\begin{align*}
		((-\Delta)^s + b_2 \cdot \nabla +  c_2)u_2 &= 0 \, \text{ in } \, \Omega,\\
		u_2|_{\Omega_e} &= u_1'|_{\Omega_e}.
	\end{align*}
By using the computation from Remark \ref{rmk:DtN_char} for $\phi \in C_c^\infty(\Omega_e \cap W_2)$ (we need the support condition here) and the hypothesis of equality of the DN maps, we get:
	\begin{align*}
		&\Lambda_{b_1, c_1}f|_{W_2 \cap \Omega_e} = \Lambda_{b_2, c_2}f|_{W_2 \cap \Omega_e}, \\
		\implies & B_1(u_1', \phi) = B_2(u_2, \phi),\\
		\implies &\int_{\mathbb{R}^n} \phi (-\Delta)^s u_1' dx= \int_{\mathbb{R}^n} \phi (-\Delta)^s u_2 dx, \\ \implies & (-\Delta)^s u_1' \equiv (-\Delta)^s u_2 \text{ on } W_2 \cap \Omega_e.
\end{align*}
Here, for convenience, we have used the abbreviation $B_{j}(\cdot, \cdot ) = B_{b_{j},c_j}(\cdot, \cdot)$ for $j\in\{1,2\}$. Furthermore, note that here we also use that $u_1'|_{\Omega_e} \in r_{\Omega_e \cap W_1} \widetilde{H}^s(W_1)$, since $\Omega' \setminus \Omega \Subset W_1$. Therefore, we have the following relations:
	\begin{equation}
		(-\Delta)^s(u_1' - u_2) = 0 \text{ on } W_2 \cap \Omega_e \quad \text{and} \quad u_1' - u_2 = 0 \text{ on } W_2 \cap \Omega_e .
	\end{equation}
	By the strong uniqueness Proposition from \ref{prop strong unique}, we conclude $u_1' \equiv u_2$ on the whole of $\mathbb{R}^n$.
	
	Now we proceed to compare the DN maps on the domain $\Omega'$. We have the following chain of equalities for $w \in C_c^\infty(W_2 \cup \Omega)$ and $u_1', u_2$ as above:
	\begin{align*}
		B_2'(u_2, w) &= \int_{\mathbb{R}^n} (-\Delta)^{\frac{s}{2}} u_2 \cdot (-\Delta)^{\frac{s}{2}} w dx + \int_{\Omega'} b_2 \cdot \nabla u_2 w dx + \int_{\Omega'} c_2 u_2 w dx\\
		&= B_2(u_2, w) + \int_{\Omega' \setminus \Omega} b_2 \cdot \nabla u_2 w dx + \int_{\Omega' \setminus \Omega} c_2 u_2 w\\
		&= B_1(u_1', w) + \int_{\Omega' \setminus \Omega} b_1 \cdot \nabla u_1' w dx + \int_{\Omega' \setminus \Omega} c_1 u_1' w\\
		&= B_1'(u_1', w).
	\end{align*}
	Here we split the integrals over $\Omega$ and $\Omega' \setminus \Omega$, used the notation $B'_1, B_2'$ for the bilinear form associated to the equation on $\Omega'$, the definition of the DN map (see Definition \ref{defi:DtN}), the fact that $u_2 = u_1'$, and $b_1 = b_2$ and $c_1 = c_2$ on $\Omega' \setminus \Omega$. 
	
	We conclude that $u_2$ solves $(-\Delta)^{s} u_2 + b_2 \cdot \nabla u_2 + c_2 u_2 = 0$ in $\Omega'$, with $u_2|_{\Omega'_e} = f|_{\Omega'_e}$. Also, we conclude that for all $w$ as above
	\[(\Lambda'_{b_2, c_2}f, w) = B_2'(u_2, w) = B_1'(u_1', w) = (\Lambda'_{b_1, c_1} f, w).\]
Here we use $\Lambda'$ notation for the DN map on $\Omega'$. The main claim follows by observing that we may pick arbitrary $w \in C_c^\infty(W_2)$. Note that we need to assume that the operators satisfy condition \eqref{eigenvalue condition} on the bigger set $\Omega'$ to assume well-definedness of the DN maps.
\end{proof}

\begin{proof}[Sketch of the argument for Proposition \ref{prop:reconstr_a}]
The proof of Proposition \ref{prop:reconstr_a} is similar to Section \ref{subsec:Finite measurements reconstruction}, so we only sketch the arguments here. With the auxiliary result of Lemma \ref{lem:ext} at hand, we deduce the claim of Proposition \ref{prop:reconstr_a} by extending $\Omega$ to $\Omega'$ suitably such that the geometric conditions on the domains are satisfied. As our operator is not self-adjoint, in the extended domain $\Omega'$ we might possibly work with Cauchy data, as we could catch a finite number of Dirichlet eigenvalues (c.f. Remark \ref{rmk:domainmonotonicity} below on how to avoid this in some cases). However, using arguments as in \cite{ruland2018lipschitz}, it is possible to deduce analogous results.
\end{proof}

\begin{rmk}[Domain monotonicity]\label{rmk:domainmonotonicity}
We remark that when the drift term $b=0$, then it is possible to avoid dealing with Cauchy data by using perturbation of domain arguments. Indeed, for the fractional Laplacian (and self-adjoint fractional Schrödinger operators), it is possible to characterize the Dirichlet spectrum through \emph{min-max formulations} \cite{FI18} (see also \cite{dFG92} for similar settings for the classical Laplacian). Relying on the weak unique continuation property, it is possible to show the monotonicity of eigenvalues 
\begin{align*}
\lambda_k(\Omega_1) >  \lambda_k(\Omega_2)\text{ for }\Omega_1 \subsetneq \Omega_2\text{ and }k\in \N,
\end{align*}
where $\lambda_k$ is the $k$-th Dirichlet eigenvalue of the fractional Laplacian. 
Thus, perturbing the domain suitably and only considering a finite number of eigenvalues in the case of self-adjoint fractional Schrödinger operators, one might work with the DN map instead of having to resort to Cauchy data.
\end{rmk}

\subsection{Proof of Theorem \ref{prop:finite_meas} and generic properties of determinants via singularity theory}

In this section, we prove the full result of Theorem \ref{prop:finite_meas}, i.e. we show that the set of exterior data from which we can choose $n + 1$ measurements in order to recover the coefficients on a compact set $K$ as in previous sections is \emph{open and dense}. This significantly improves the result from Section \ref{subsec:Finite measurements reconstruction} in that the data still depend on the \emph{unknown} potentials $b,c$, but in a precise sense they form a large set, i.e. given (random) exterior data, we know that an arbitrarily small perturbation of them will render them admissible in our reconstruction scheme.

The main idea of our argument is to relax the condition that for admissible exterior data $f_1,\dots,f_{n+1}\in C_c^{\infty}(W)$ we require
\begin{align*}
h(P_{b,c}(f_1),\dots,P_{b,c}(f_{n+1})) \neq 0 \mbox{ in } K \subset \Omega, 
\end{align*}
where $h$ was the function from \eqref{eq:h}. Instead, we consider data $f_1,\dots,f_{n+1}\in C_c^{\infty}(W)$ such that $h(P_{b,c}(f_1),\dots,P_{b,c}(f_{n+1}))$ is only allowed to vanish to a finite (dimension-dependent) order. Then, by known results from \cite[Lemma 3]{bar1999zero}, it follows that the set  
$$
\left\{x\in K; \ h(P_{b,c}(f_1),\dots,P_{b,c}(f_{n+1}))(x)=0\right\}
$$ 
is of measure zero in $K \Subset \Omega$.\footnote{More specifically, it is countably-$C^\infty$-rectifiable, i.e. covered by a countable union of hypersurfaces.} As a consequence, due to the continuity of $b$, $c$, it is then possible to reconstruct both coefficients (see Lemma \ref{lem:aux_recov}). Simultaneously, the set of exterior data $f_1,\dots,f_{n+1}\in C_c^{\infty}(W)$ for which $h(P_{b,c}(f_1),\dots,P_{b,c}(f_{n+1}))$ vanishes only of finite (dimension-dependent) order immediately by definition is \emph{open} in $C_c^{\infty}(W)^{n+1}$. The \emph{density} of such data will be obtained via small perturbations, relying on ideas of Whitney's work \cite{whitney}, which had been developed in the context of \emph{singularity theory}. Technically, this is the most involved part of our arguments.

Let us introduce the set of our admissible exterior conditions. 

\begin{defi}
\label{defi:data} Let $b$, $c$ satisfy the conditions in Theorem \ref{prop:finite_meas}.
Let $n\in \N$ and $k(n)=  \Big\lceil{\sqrt{n + 1}}\Big\rceil \in \N$ (i.e. $k(n)$ is the smallest positive integer greater or equal to $\sqrt{n + 1}$). Let $\Omega \subset \R^n$ be as in Theorem \ref{prop:finite_meas} and $K\Subset \Omega$ be a compact set as in Section \ref{subsec:Finite measurements reconstruction}.
Then, we define the set $\mathcal{F} \subset C_c^{\infty}(W)^{n + 1}$ to be the set of exterior data, such that $(f_1, \dotso, f_{n + 1}) \in \F$, if $h(P_{b,c} f_1, \dotso, P_{b, c} f_{n + 1})$ has at most order of vanishing $k(n)-2$ at each point of $K \Subset \Omega$.
\end{defi}

We claim that the set $\F$ yields the desired set of exterior data for the proof of Theorem \ref{prop:finite_meas}. To this end, we first show that given exterior data $(f_1, \dotso, f_{n + 1})\in \F$, it is possible to recover $b$ and $c$:

\begin{lem}
\label{lem:aux_recov}
Assume that the conditions of Theorem \ref{prop:finite_meas} hold.
Let $(f_1, \dotso, f_{n + 1})\in \F$. Then it is possible reconstruct $b,c$ from the exterior measurements of $f_l$ and $\Lambda_{b,c}(f_l)$, for $l \in \{1,\dotso,n+1\}$.
\end{lem} 

\begin{proof}
We first recall that the zero set of a smooth function not vanishing to infinite order is contained in a countable union of codimension one submanifolds (see e.g. \cite[Lemma 3]{bar1999zero}). In particular, by definition of the set $\F$, we thus infer that $h(P_{b,c}(f_1),\dots,P_{b,c}(f_{n+1}))$ vanishes only on a set of Lebesgue measure zero. Let us denote this measure zero set by $B \subset K$.
Therefore, the argument in \eqref{eq:reconstruction} goes through on the set $K \setminus B$. But $b, c$ satisfy $\supp(b)\cup \supp(c)\Subset K$ and are smooth, so have unique continuous extensions to $K$, which can be determined from $b|_{K\setminus B}$ and $c|_{K \setminus B}$.
This concludes the argument for the reconstruction of $b,c$ from $f_l$ and $\Lambda_{b,c}(f_l)$ for $l\in \{1,\dotso,n+1\}$.
\end{proof}

Hence, it remains to prove the \emph{openness} and \emph{density} of the set $\F\subset C^\infty_0(W)^{n+1}$. While the \emph{openness} is a direct consequence of the definition of the set $\F$, the \emph{density} of the set $\F$ requires careful arguments. This will be the content of the remaining subsections.

\subsubsection{Generic properties of determinants via singularity theory} 

In order to deduce the density of the set $\F$, we seek to argue by perturbation: The main idea is that for an $m$-tuple of functions $f = (f_1, \dotso, f_m)$ on $\mathbb{R}^n$, and some differential relation $P(x,D)(f)(x)=0$ on $\mathbb{R}^n$, we may generate a parametric family $f_{\alpha}$, in such a way that on a compact set, near any $\alpha_0$ there is an $\alpha$ arbitrarily close, such that $P(x,D)(f_\alpha) \neq 0$ on $K$. Here $\alpha \in \mathbb{R}^N$ for some (large) $N \in \N$ and $f_\alpha$ is given by adding a polynomial of degree $r$ (related to $N$) to each of the entries $f_i$ of $f$, with coefficients given by reading off indices of $\alpha$ in a suitable order.

A famous example, due to Morse, of this fact is just a $C^2$ function $f: \mathbb{R}^n \to \mathbb{R}$. Then by looking at $f_\alpha (x) = f(x) + \langle{\alpha, x}\rangle$, where $\langle{\cdot, \cdot}\rangle$ is the inner product and $\alpha \in \mathbb{R}^n$, by \emph{Sard's theorem} the set of $\alpha$ for which $f_\alpha$ has a degenerate critical point is of measure zero. These ideas were generalised by Whitney \cite{whitney} and others in the area of singularity theory to study generic maps $\mathbb{R}^n \to \mathbb{R}^m$.

In our case, the idea is that by adding generic polynomials of degree $k(n)\in \N$ with small coefficients to $P_{b, c} f_1, \dotso, P_{b, c} f_{n + 1}$, we may obtain perturbations such that the determinant function $h$ from \eqref{eq:h} only vanishes of order at most $k(n)-2$. Here $k(n)\in \N$ is the constant from Definition \ref{defi:data}.
Since the perturbation is just by polynomials of order $k(n)$, i.e. by a linear combination of one of 
\begin{align}\label{eq:N_0}
N_0 =  \sum\limits_{j=0}^{k(n)} {n +  j - 1 \choose j}
=  {n +k(n) \choose k(n)}
\end{align}
linearly independent polynomials of the form $1, x_i, x_i x_j, ...$, by Runge approximation we may approximate these by $P_{b, c} (f_{i,m})$ for $i\in \{1 , \dotso, N_0\}$ arbitrarily close, for some suitable exterior data $\left\{f_{i,m}\right\}$'s. By adding a linear combination of $f_{i,m}$ with coefficients $\alpha$ to the exterior data, one can obtain an arbitrary close measurement for which the zero set of $h$ is ``good", i.e. is just given by a stratification of smooth hypersurfaces (see Propositions \ref{prop:n=1} and \ref{prop:n} below). In the sequel, we present the details of this argument.

\subsubsection{Preliminaries}
We need to import some (old) technology from \cite[Parts A and B]{whitney}, which allows us to modify functions in a favorable way by simply applying dimension-counting arguments. We consider a mapping $f = (f_1, \dotso, f_m): \mathbb{R}^n \to \mathbb{R}^m$. We consider derivatives of order up to $r \in \mathbb{N}$ of $f$ and a map $\bar{f}: \mathbb{R}^n \to \mathbb{R}^N$ given by arranging the partial derivatives of $f$ in some fixed order. Then there is a ``bad" set $S \subset \mathbb{R}^N$ that we would like to avoid. In general, the space $S$ is stratified, i.e. there is a splitting $S = \cup_{i \leq \mu} S_i$, where $S_i$ are smooth manifolds of dimension $\dim S_i$ for $i\in \{1,\dotso,\mu\}$. More precisely, we say $S$ is a \emph{manifold collection of defect $\delta$}, if $\cup_{i \leq j} S_i$ is closed for all $j\in \{1,\dotso,\mu\}$ and $\codim S_i \leq \delta$ for all $i$.

We alter $f$ by adding to it a polynomial of degree $\leq r$ whose coefficients form a set $\alpha$ of very small numbers. If $f_{\alpha}$ is the resulting mapping $\mathbb{R}^n \to \mathbb{R}^N$, we may prove that for compact subsets $K \subset \mathbb{R}^n$ and $T \subset S$ there exists $\alpha$ arbitrarily small such that $f_{\alpha}(K) \cap T = \emptyset$. If we fix $K$, by taking $\bar{f}(K) \subset B_L \subset \mathbb{R}^N$ for some large $L$, we prove that for any $K$, there is an $\alpha$ such that $f_{\alpha}(K) \cap S = \emptyset$.

Let $N \in \N$ be the number given as above. Assume that we have a smooth map for $(p, \alpha) \in \Omega \times R_1 \subset \mathbb{R}^n \times \mathbb{R}^N$
\[F(p, \alpha) = f_\alpha (p) = f_p^*(\alpha)\]
into $\R^N$.
The family $f_{\alpha}$ is called an \emph{$N$-parameter family} of mappings of $\Omega $ into $\mathbb{R}^N$ if the matrix $\nabla_\alpha f_p^*(\alpha)$ has full rank for all $p \in \Omega$.

Next, we say a subset $Q$ of a Euclidean space is of \emph{finite $\mu$-extent} if there is a number
$A$ with the following property. For any integer $\kappa$, there are sets $Q_1, \dotso, Q_a$ such that
\begin{align}
Q = Q_1 \cup \dotso \cup Q_a, \quad \diam(Q_j) < 1/2^\kappa \ (\text{for all } j),\quad a \leq 2^{\mu \kappa}A.
\end{align}

\begin{lem}[Lemma 9a in \cite{whitney}]\label{lem:9a} Let the $f_{\alpha}$ form an $N$-parameter family of mappings of $\Omega \subset \mathbb{R}^n$ into $\mathbb{R}^N$ and let $\Omega_1\subset \Omega$ and $Q$ be compact subsets of $\mathbb{R}^n$ and of $\mathbb{R}^N$ of finite $\omega$-extent and of finite $q$-extent, respectively, and suppose $\omega +q < N$. Then for any $\alpha_0\in \R^N$ and for any $\epsilon>0$ there exists $\alpha$ with $|\alpha-\alpha_0|<\epsilon$ such that 
 \[f_\alpha(\Omega_1) \cap Q = 0.\]
\end{lem}

By \cite[Section 10]{whitney}, the family of functions $f_\alpha$ given by adding polynomials of order $\leq r$ is an $N$-parameter family of mapping of $\mathbb{R}^n$ into $\mathbb{R}^N$. 
We call $\delta := N - q$ the \emph{defect} of $S$ in $\mathbb{R}^N$, so the condition in Lemma \ref{lem:9a} can be restated as simply $\delta > \omega$. It can be easily seen that the conclusion of the above lemma holds if $Q$ is a manifold collection of defect $\delta > \omega$.

\subsubsection{Density argument}

Let $\mathcal{F} \subset C_c^{\infty}(W)^{n + 1}$ be the set, which contains exterior measurements from Definition \ref{defi:data}. By Lemma \ref{lem:aux_recov}, we may reconstruct the drift $b$ and potential $c$ for such exterior measurements. 

Notice that $\mathcal{F}$ is an \emph{open} set, since the set of all functions $g$ with finite order of vanishing at each point is open and the operator $P_{b, c}$ is continuous in given topologies. In the sequel, we seek to prove that $\mathcal{F}$ is also \emph{dense}.

We first illustrate the argument for the case $n=1$ in which case it is rather transparent. We will then present the proof for the general case below.

\begin{prop}\label{prop:n=1}
	Assume $n = 1$. Then $\mathcal{F} \subset C_c^{\infty}(W)^2$ is open and dense.
\end{prop}

\begin{proof}
	Take $f = (f_1, f_2) \in C_c^{\infty}(W)^2$ to be any exterior data and consider $g_i := P_{b, c}f_i \in C^\infty(K)$ for $i = 1, 2$. Write $g = (g_1, g_2)$. We will construct an approximation of $(f_1, f_2)$ in the topology of $C_c^{\infty}(W)^2$ lying in $\mathcal{F}$. In one dimension, we have
	\begin{align}\label{eq:diffreln}
	\begin{split}
	h(g_1, g_2) = \begin{pmatrix} g_1' & g_1\\
	 g_2' & g_2
	\end{pmatrix} = g_1' g_2 - g_1 g_2'.
	\end{split}
	\end{align}
	In other words, $h$ is the Wronskian of $g_1$ and $g_2$. Now, in this case, we have
	\[\bar{g}(x) = \big(g_1(x), g_2(x), g_1'(x), g_2'(x), g_1''(x), g_2''(x)\big)\]
	and $N = 2N_0= 6$ (where $N_0$ was defined in \eqref{eq:N_0}). We consider perturbations of the form
	\begin{align*}
		g_\alpha(x) = g(x) + (\alpha_1 + \alpha_1' x + \alpha_1'' x^2)e_1 + (\alpha_2 + \alpha_2' x + \alpha_2'' x^2)e_2 \in \R^2,
	\end{align*}
	where $e_1, e_2$ are the canonical coordinate vectors in $\R^2$.

	We compute the defect of the bad set given by $h=0$ and $\nabla h = 0$, i.e. 
	\begin{align*}
		S := \{J_1 = \alpha_1' \alpha_2 - \alpha_1 \alpha_2' = 0\} \cap \{J_2 = \alpha_1'' \alpha_2 - \alpha_1 \alpha_2'' =0 \} \subset \mathbb{R}^6.
	\end{align*}
	We need to show that the defect $\delta = 6 - q$, where $q$ is the extent of $S$, is bigger that $n = 1$ to apply Lemma \ref{lem:9a}. To this end, we compute the gradients
	\begin{align*}
		\nabla \bar{J}_1 = (-\alpha_2', \alpha_1', \alpha_2, -\alpha_1, 0, 0),\\
		\nabla \bar{J}_2 = (-\alpha_2'', \alpha_1'', 0, 0, \alpha_2, -\alpha_1).
	\end{align*}
	These are clearly linearly independent for $\alpha_1 \neq 0$ or $\alpha_2 \neq 0$ or $\det \begin{pmatrix}
	\alpha_1' & \alpha_2'\\
	\alpha_1'' & \alpha_2''
	\end{pmatrix} \neq 0$, so $S$ is a manifold collection of defect $\delta = 2$. So we apply Lemma \ref{lem:9a} to obtain arbitrarily small values of $\alpha = (\alpha_1, \alpha_1', \alpha_1'', \alpha_2, \alpha_2', \alpha_2'')$ such that $g_\alpha$ satisfies the property that $h(g_\alpha)$ has an empty critical zero set on $K \subset \overline{\Omega}$.
	
	Next, by Runge approximation (see Lemma \ref{lem:Higher Runge}), there exists $f_{0, m}, f_{1, m}, f_{2, m} \in C_c^{\infty}(W)$ with $P_{b, c} f_{i, m} \to x^i$ in $C^\infty(K)$ for $i = 0,1, 2$ and as $m \to \infty$. Therefore, we define
	\[f_{\alpha, m} = f + (\alpha_1 f_{0, m} + \alpha_1' f_{1, m} + \alpha_1'' f_{2, m}) e_1 + (\alpha_2 f_{0, m} + \alpha_2' f_{1, m} + \alpha_2'' f_{2, m}) e_2 \in \R^2. \]
	Now fix $m$ large enough, so $P_{b, c} f_{i, m}$ is close to $x^i$, such that the perturbation by elements $f_{i, m}$ of $f$ makes a $6$-parameter family of mappings $\Omega' \to \mathbb{R}^6$, for $\alpha \in B_1(0) \subset \mathbb{R}^6$ in the unit ball (say), where $K \subset \Omega' \Subset \Omega$, by compactness. Then we again apply Lemma \ref{lem:9a} and get that $g_{\alpha,m} := P_{b, c} f_{\alpha, m} \to g$ in $C^\infty(K)$ on a sequence of $\alpha$ converging to zero. By construction $f_{\alpha, m} \in \mathcal{F}$, so this finishes the proof.
\end{proof}

\begin{rmk}
\label{rmk:generic}
To prove the desired genericity property, we need to approximate $P_{b, c} f$ by either polynomials or other nice functions (analytic, generic etc.), by use of a linear approximation operator $T_m f$, but such that
\begin{align*}
	T_m P_{b, c} f = P_{b, c} T_m f \quad \text{and} \quad \lim_{m \to \infty} T_m f = f.
\end{align*}
This is the reason why the usual approximation operators, such is the Bernstein polynomials operator, or the general Weierstrass approximation theorem approach are not good for this purpose. The above approximation argument however proves the existence of such $T_m$, which is obtained by adding a finite linear combination of suitable functions with coefficients going to zero as $m \to \infty$.
\end{rmk}

\begin{rmk}
\label{rmk:Morse}
There is an alternative proof by hand of the above statement for $n = 1$, not using the Whitney machinery, but only the genericity of Morse functions.
\end{rmk}

We seek to extend the previous argument to dimension $n=2$ and higher. Unfortunately, in this context, it does not suffice to consider the critical zero set of $h$, i.e. the set $h=0$ and $\nabla h=0$. The computations below show that we have to include higher order derivatives of $h$ to obtain genericity in the sense of the previous proposition. 

Let $N = (n + 1) \times {n + k(n) \choose k(n)}$ be the number of polynomials of degree $\leq k(n)$ with which we perturb (we multiply by $n + 1$ as this is our number of functions). Then we consider the map $\bar{g}:\mathbb{R}^n \to \mathbb{R}^N$ (see previous subsection) of evaluating the derivatives of order $\leq k(n)$ at each point. We define the bad set to be $S = S_{k(n)} \subset \mathbb{R}^N$ consisting of points given by the condition that $h$ and its derivatives up to order $k(n) - 1$ vanish. For a given function $g = (g_1, \dotso, g_{n + 1})$, we denote this set by $Z_0(g) = h^{-1}(0)$, $Z_1(g) = Z_0 \cap (\nabla h)^{-1} (0)$, and inductively we define $Z_j(g) = Z_{j-1}(g) \cap (\nabla^j h)^{-1}(0)$.

Similarly as in one dimension, we then also have the following lemma, which we prove in the appendix:

\begin{lem}[Determinant genericity]\label{lem:genericity}
	The bad set $S_{k(n)} \subset \mathbb{R}^N$ is a manifold collection of defect $n + 1$. In particular, there are arbitrarily small perturbations $g_\alpha$ with $\alpha \in \mathbb{R}^N$, such that $Z_{k(n) - 1}(g_{\alpha}) = \emptyset$ on an arbitrary compact set.
\end{lem}

As a corollary, we deduce the main result.

\begin{prop}	
\label{prop:n}
	For any $n \in \mathbb{N}$, the set $\mathcal{F} \subset C_c^{\infty}(W)^{n + 1}$ is open and dense.
\end{prop}

\begin{proof}
	The proof is an immediate corollary of Lemma \ref{lem:genericity} and the method of proof of Proposition \ref{prop:n=1}. Indeed, openness again follows from the definition of the set $\F$. In order to infer the density of $\F$, we argue along the lines of Proposition \ref{prop:n=1}: Let $f=(f_1,\dots,f_{n+1})\in C_c^{\infty}(W)^{n+1}$ be arbitrary but fixed. 
By Runge approximation (see Theorem \ref{thm: Runge approximation} (b)), for each $\beta \in \N^{n}$ with $|\beta|\leq \lceil{\sqrt{n + 1}}\Big\rceil$ there exists $f_{\beta,m} \in C_c^{\infty}(W)$ such that 
	\begin{align*}
		P_{b,c}(f_{\beta,m}) \rightarrow x^{\beta} \mbox{ in } C^{\infty}(K).
	\end{align*}
As the set of polynomials up to degree $|\beta|\leq \Big\lceil{\sqrt{n + 1}\Big\rceil}$ forms a $\nu$-parameter family with $\nu = N_0$ and $N_0$ as in \eqref{eq:N_0}, for $m\geq m_0$ sufficiently large, also the set $P_{b, c} (f_{\beta,m})$ with $|\beta| \leq \Big\lceil{\sqrt{n + 1}}\Big\rceil $ forms a $N_0$-parameter family. As a result, Lemma \ref{lem:9a} can be applied to
	\begin{align*}
		P_{b,c}(f_{\alpha}):= P_{b,c}(f) + \sum\limits_{j=1}^{ n + 1} \sum\limits_{\beta \in \N^n, \ |\beta| \leq \Big\lceil{\sqrt{n + 1}}\Big\rceil} \alpha_{\beta j} P_{b,c}(f_{\beta, m}) e_j,
	\end{align*}
where $\{e_1, \dots, e_{n+1}\}$ denotes the canonical basis of $\R^{n + 1}$, $\alpha_{\beta j} \in \R$ and $x^{\beta}:= \prod\limits_{\ell=1}^{n} x_{\ell}^{\beta_{\ell}}$. Thus, for any $\epsilon>0$ there exists $\alpha_{\epsilon} \in \R^{N_0  (n +1)}$ with $|\alpha_{\epsilon}|\leq \epsilon$ such that $Z_{ k(n) - 1}(P_{b,c}(f_{\alpha_{\epsilon}})) = \emptyset$ on $K \subset \Omega$. By construction, we have $f_{\alpha_{\epsilon}} \in \mathcal{\F}$. This concludes the density proof.
\end{proof}

Combining Lemma \ref{lem:aux_recov} and Propositions \ref{prop:n=1}, \ref{prop:n} then implies the result of Theorem \ref{prop:finite_meas}.

\begin{rmk}
We remark that Theorem \ref{prop:finite_meas} together with the results from \cite{GRSU18} also yields a constructive reconstruction algorithm for the fractional Calder\'on problem with drift.
\end{rmk}

\begin{rmk}
Last but not least, we point out that similar openness and density results can also be obtained for the Jacobian by arguing along the same lines as in the Appendix. More specifically, genericity results in Lemma \ref{lem:genericity} can be shown to hold with the same critical index $k(n) - 1$, if instead of the determinant in equation \eqref{eq:h} we consider the Jacobian determinant of $n$ functions, which might be of independent interest.
\end{rmk}

\appendix
\section{Genericity of determinants}
\label{sec:appendix}

The aim of this appendix is to prove Lemma \ref{lem:genericity}. In order to determine the conditions which are imposed by the sets $Z_i$ from above, we first rewrite the condition $\nabla h = 0$ in closed form. Using that $\frac{\p \det(M)}{\p M_{ij}} = \Cof(M)_{ij}$ for a matrix $M$, where $\Cof(M)$ denotes the matrix of cofactors, we obtain 
\begin{align*}
\p_{x_j} h(g_1,\dots,g_{n+1})
= \frac{\p h(M)}{\p M_{kl}} \frac{\p M_{kl}}{\p x_{j}}
= \Cof(M)_{kl} \frac{\p M_{kl}}{\p x_{j}},
\end{align*}
where we use summation convention of repeated indices, and  
\begin{align*}
M=M(g_1,\dots,g_{n+1})
= \begin{pmatrix}
\frac{\p g_1}{\p x_1} & \hdots & \frac{\p g_1}{\p x_n} & g_1\\
\vdots & \hdots & \vdots & \vdots \\
\frac{\p g_{n+1}}{\p x_1} & \hdots & \frac{\p g_{n+1}}{\p x_{n}} & g_{n+1} 
\end{pmatrix}
\end{align*}
abbreviates the entries of $h$.
Hence, by the column-wise expansion of the determinant the condition $\p_{j}h=0$ can be reformulated as 
\begin{align}\label{eq:derivative}
\sum\limits_{l=1}^{n+1} \det(\widetilde{M}_l^j) = 0 \mbox{ for } j\in \{1,\dots, n\},
\end{align}
where 
\begin{align*}
\widetilde{M}_l^j = M\left(g_1, \dots, g_{l-1}, \frac{\p g_l}{\p x_j}, g_{l+1}, \dots, g_{n+1}\right).
\end{align*}
Higher order derivatives can be computed analogously (see Section \ref{sec:corank2}).

Hence, in the formal variables $\alpha \in \R^{N_0}$ the two constraints associated with $Z_1$ turn into the formal constraints
\begin{align}
\label{eq:a}
\det(M(\alpha)) &= 0, \\ 
\label{eq:b}
\sum\limits_{l=1}^{n+1} \det(\widetilde{M}_{l}^j(\alpha)) &= 0 \mbox{ for } j =1,\dots,n,
\end{align}
where $\alpha \in \R^{N_0}$ denotes the vector of the component mapping and 
\begin{align*}
M(\alpha) = \begin{pmatrix}
\alpha^1_1 & \alpha_2^1 & \hdots & \alpha_{n}^1 & \alpha^1\\
\alpha_1^2 & \alpha_2^2 & \hdots & \alpha_n^2 & \alpha^2\\
\vdots & \vdots & \vdots & \vdots & \vdots\\
\alpha_1^{n+1} & \alpha_2^{n+1} & \hdots & \alpha_{n}^{n+1} & \alpha^{n+1}
\end{pmatrix}.
\end{align*}
Here we have used the convention that in the arrangement of the derivatives defining $\alpha \in \R^{N_0}$ in \eqref{eq:N_0} we have set $\alpha^{j}$ to correspond to the function $g_j$, $j\in\{1,\dots,n+1\}$, and $\alpha^{j}_k$ corresponding to the partial derivatives $\frac{\p g_j}{\p x_k}$, $j\in\{1,\dots,n+1\}$, $k\in\{1,\dots,n\}$.

As already indicated above, we note that in general the conditions imposed by $Z_1$ do not suffice to apply Lemma \ref{lem:9a}: Indeed, in order to apply Lemma \ref{lem:9a} we have to prove that the bad set is of co-dimension $n+1$. When $n \geq 4$ the set of matrices with co-rank equal to two for instance is non-negligible (i.e. it has larger co-dimension than $n+1$; where we recall that the set of $m_1 \times m_2$ matrices of rank $r$ is a manifold of co-dimension equal to $(m_1 - r) (m_2 - r)$). However, if $\codim M(\alpha) = 2$, then the equations in \eqref{eq:b} are void, i.e. are automatically implied by the corank condition (by definition of the cofactor matrix). Thus, in order to arrive at a stratification of the bad set with a sufficiently large co-dimension, for $n \geq 4$ we \emph{must} include higher order derivatives of the determinant $h$; these are encoded in the sets $Z_i$.

The proof of Lemma \ref{lem:genericity} will be an induction on the corank of $M(\alpha)$, which we discuss in the next three sections.

\subsection{Corank one case} We analyse the case when the corank of the matrix $M(\alpha)$ is equal to one and prove a stratification into a collection of codimension at least $n + 1$ submanifolds in that case. We will use the notation $\alpha^i$ to represent the parameter corresponding to the function $f_i$. The derivatives $\partial_{jk} f_i$ correspond to $\alpha^i_{jk} = \alpha^i_{kj}$ etc. Assume $\corank M(\alpha) = 1$. \\

\emph{Step 1.} Assume $\Cof(M)_{11} \neq 0$. By the expansion in \eqref{eq:b} for $j = 1$ and the Laplace expansion of the determinant, we get an equation for $\alpha^{1}_{11}$. Since this is the unique equation for $\alpha^{1}_{11}$ (the relations \eqref{eq:b} for $j = 2, \dotso, n$ do not contain information on $\alpha_{11}^1$), it is automatically linearly independent of any other relation. We denote the gradient of equation \eqref{eq:b} for $j = 1$ by $G_1$. Denote by $G_j$ the gradient of equation \eqref{eq:b} applied to $j$. We claim that the gradients $G_1, \dotso, G_n$ are linearly independent.

From \eqref{eq:b} for $j = 2$, we get an equation for $\alpha^{1}_{12}$ (as $\Cof(M)_{11} \neq 0$). The gradient $G_2$ of this equation is linearly independent from $G_1$ since this equation doesn't contain $\alpha_{11}^1$. 

Similarly, from \eqref{eq:b} for $j = 3$, we get an equation for $\alpha^{1}_{13}$ (as $\Cof(M)_{11} \neq 0$). The gradient $G_3$ of this equation is linearly independent of $G_1$ and $G_2$, since $G_1$ is the only one containing a non-zero component at the place of $\alpha^1_{11}$. Then, the component of $G_2$ corresponding to $\alpha^{1}_{13}$ is zero, so we obtain the claimed independence.

Iteratively, we obtain that for any $j=1, \dotso, n$, we can extract $\alpha^{1}_{1j}$ from \eqref{eq:b} applied to $j$, and the so obtained gradient $G_j$ is linearly independent from $G_1, \dotso, G_{j-1}$ by a similar argument as in previous two paragraphs.

It is left to observe that the gradient $G$ of the equation \eqref{eq:a} is independent of $G_1, \dotso, G_n$, since \eqref{eq:a} doesn't contain information about $\alpha^{1}_{1j}$ for any $ j=1,\dotso, n$. Therefore, we are contained in a codimension $n + 1$ set.\\

\emph{Step 2.} Assume $\Cof(M)_{11} = 0$, but $\Cof(M)_{12} \neq 0$. Similarly as before, by \eqref{eq:b} for $j = 1$ and the Laplace expansion of the determinant, we get an equation for $\alpha^{1}_{12}$ (as $\Cof(M)_{12} \neq 0$). Denote the gradient of this equation by $G_1$. Also, denote by $G_j$ the gradient of \eqref{eq:b} for $j$. We claim that $G_1, \dotso, G_n$ are linearly independent.

From \eqref{eq:b} for $j = 2$, we get an equation for $\alpha^{1}_{22}$ (as $\Cof(M)_{12} \neq 0$). The gradient $G_2$ of this equation is linearly independent of $G_1$, since the component of $G_1$ at the place of $\alpha^{1}_{22}$ is equal to zero.
Similarly, from \eqref{eq:b} for $j = 3$, we get an equation for $\alpha^{1}_{23}$ (as $\Cof(M)_{12} \neq 0$). The gradient $G_3$ of this equation is linearly independent of $G_1$ and $G_2$. This is because $G_1$ and $G_3$ do not contain information about $\alpha^1_{22}$; then because the equation corresponding to $G_1$ contains no information about $\alpha^{1}_{32}$.

Iterating (as before) we obtain linearly independent gradients $G_1, \dotso, G_n$. It is left to observe that the gradient $G$ of equation \eqref{eq:a} is independent of $G_1, \dotso, G_n$, since \eqref{eq:a} does not contain information about $\alpha^{1}_{2,j}$ for any $ j=1,\dotso, n$, by looking at the entry $\alpha^1_2$. Therefore, we are contained in a codimension $n + 1$ set.\\

\emph{Step 3.} We assume $\Cof(M)_{ij} = 0$ for $i=1,\dotso,n + 1$ and $j=1,\dotso, n$ (otherwise the arguments are reduced to previous two steps). Then without loss of generality (w.l.o.g.), assume $\Cof(M)_{1, n + 1} \neq 0$ (we know that $\Cof(M)$ has rank equal to $1$ by the condition $\corank(M) = 1$). Then from \eqref{eq:b}, we get $n$ independent relations by looking at the coefficient next to $\alpha^1_{j}$ for $j = 1, \dotso, n$, i.e. the gradients $G_1, \dotso, G_n$ of these equations are linearly independent.

Now we are just left to observe that the components of the gradient $G$ of \eqref{eq:a} corresponding to $\alpha^1_i$ are equal to zero by the cofactor condition. Thus $G$ is linearly independent of $G_1, \dotso, G_n$, by looking at the coefficient of $\alpha^1$. 

\subsection{Corank two case}
\label{sec:corank2}
Assume now $\corank(M) = 2$. We need to \emph{include the second order derivatives}, unless $ n<4$. This is because matrices of $\corank M = 2$ have codimension $4$, and the equations \eqref{eq:b} are automatically satisfied in the co-rank two case. The equations for the second order derivatives read
\begin{align}\label{eq:c}
	\sum_{k, l = 1}^{n+1} \det \big(\widetilde{M}_{kl}^{ij}(\alpha)\big) = 0, \quad \text{
for each } \quad i, j =1,\dotso, n.
\end{align}
Here we write $\widetilde{M}^{ij}_{kl}(\alpha)$ for the matrix
\begin{align*}
	\widetilde{M}^{ij}_{kl} = M\left(g_1, \dots, g_{k-1}, \frac{\p g_k}{\p x_i}, g_{k+1}, \dots, g_{l-1}, \frac{\p g_l}{\p x_j}, g_{l+1}, \dots, g_{n+1}\right).
\end{align*}
Here w.l.o.g. we may assume $k < l$. If $k = l$, then we get
\begin{align*}
	\widetilde{M}^{ij}_{kk} = M\left(g_1, \dots, g_{k-1}, \frac{\p^2 g_k}{\p x_i \p x_j}, g_{k+1}, \dots, g_{n+1}\right).
\end{align*}

Since $\corank M(\alpha) = 2$, it is easy to see that $\det \widetilde{M}^{ij}_{kk}(\alpha) = 0$ and we only need to consider matrices with first order derivatives (even in the higher corank case) in the equation \eqref{eq:c}.\\

\emph{Main argument.} W.l.o.g., we assume that the matrix $M_0$ is obtained by erasing the first row and the first column of $M(\alpha)$ has corank equal to one (other cases are dealt with similarly as in Step 2 above). Consider equations \eqref{eq:c} for $i = 1$ and $ j=2,\dotso, n$. By the standard cofactor expansion of the determinant, the coefficient next to $\alpha^1_{11}$ is equal to
\begin{align}\label{eq:coeff1}
	C_j = \sum_{k = 2}^{n+1} \det \widetilde{M}_{1, k}^{j}(\alpha).
\end{align}
Here the matrix $\widetilde{M}_{1, k}^{j}(\alpha)$ is obtained by erasing the first row, i.e.
\begin{align*}
	\widetilde{M}_{1, k}^j = M\left(g_2, \dots, g_{k-1}, \frac{\p g_k}{\p x_j}, g_{k+1}, \dots, g_{n+1}\right).
\end{align*}\\

\emph{Case 1.} Assume we have $C_j = 0$ for each $j=2,\dotso, n$. In this case the $n$ equations in \eqref{eq:coeff1} together with the condition that $\det(M_0)=0$ are exactly of the form as in the corank one case, which had been discussed in the previous section. Hence, by the inductive hypothesis (i.e. by the $\corank M(\alpha) = 1$ case), we have that gradients of $C_j$, together with the gradient of $\det M_0$ span an $n$-dimensional space.

Now putting the first row and column of $M(\alpha)$ back to $M_0$, by the corank condition on $M(\alpha)$, we get additionally at least one more independent relation 
and so this totals to $n + 1$ independent gradients. This means that the co-dimension of the bad set is at least $n + 1$.\\

\emph{Case 2.} Assume w.l.o.g. $C_2 \neq 0$. Similarly to Steps 1 and 2 above, we consider the gradients of the equations \eqref{eq:c} for a range of indices $i = 1, \dotso, n$ and $j = 2$. Denote these gradients by $G_1, \dotso, G_n$. We claim that $G_1, \dotso, G_n$ are linearly independent. Now the gradient $G_n$ is non-zero at the component of $\alpha^1_{n1}$, with the coefficients $C_2$. Note that the component of $\alpha^1_{n1}$ in $G_2, \dotso, G_{n - 1}$ is equal to zero. Also, the component of $\alpha_{11}^1$ is equal to zero in $G_2, \dotso, G_n$ and so $G_n$ is linearly independent of $G_1, \dotso, G_{n-1}$.

Similarly, for each $ k=1,\dotso,n$, the component of $\alpha^{1}_{k1}$ in $G_k$ is equal to $C_2$ and the gradients $G_2, \dotso, G_{k - 1}$ have the component of $\alpha^{1}_{k1}$ equal to zero. But the component of $\alpha_{11}^1$ in $G_2, \dotso, G_k$ is equal to zero. We therefore inductively conclude that $G_k$ is linearly independent of $G_1, \dotso, G_{k - 1}$. This proves the claim.

As a consequence, we obtain $n$ linearly independent conditions from the gradients $G_1,\dots,G_n$ which put constraints on the components $\alpha_{j1}^1$ with $j =1,\dotso, n$. Combined with the condition from $\det(M(\alpha))=0$ (which does not involve any component of $\alpha$ with two indeces) this yields that the bad set is at least of co-dimension $n+1$.

\subsection{Corank at least three} The argument from the corank two case is now easy to generalise. For the case that corank of $M$ is equal to $k$, we need to include $k$ derivatives of $\det M$. The extra equations obtained are analogous to \eqref{eq:c} and now we have a similar equation for each $k$-tuple of integers $1 \leq i_1, i_2, \dotso, i_k \leq n$. Again these extra equations only involve $\alpha^{i}_{jk}$ (and not $\alpha$'s corresponding to higher derivatives), due to the corank condition on $M(\alpha)$.

The cases 1 and 2 above are then obtained in a similar fashion. The case 1 always uses the induction hypothesis in the setting of $n-1$ variables and $k-1$ (first order) derivatives (the equation for $Z_k$ always involves $k$ first order derivatives of $g_1,\dots,g_{n+1}$ of which we always delete one of the rows with one of the first order derivatives defining quantities analogous to $C_j$ from \eqref{eq:coeff1}). Thus, the fact that the bad set is of co-dimension at least one is always a consequence of the co-dimension estimate which follows from the setting with $n-1$ variables and $k-1$ (first order) derivatives (which was proved in the step before) and yields $n$ co-dimensions. Adding the co-dimension from the deletion of the additional row thus gives $n+1$ co-dimensions as desired.

The case 2 is the induction step and is carried out in a very similar way as outlined above, i.e. we use the condition obtained from $Z_k$ involving $k$ first order derivatives of $g_j$. Arguing in this way yields at least $n + 1$ independent relations. We can see that, under the assumption $\corank M(\alpha) = k$, the co-dimension of the bad set should be roughly $n + k^2$, but this is not important for the main argument, as long as we have co-dimension at least $n + 1$.\\

This concludes the argument for the estimate on the co-dimension of the bad set if one considers $Z_1,\dots,Z_k$ with 
\[k = \Big\lceil{\sqrt{n + 1}}\Big\rceil  -1, \]
as for $ \corank M(\alpha) \geq \Big\lceil{\sqrt{n + 1}}\Big\rceil $ already by the linear dependence relations we obtain the desired co-dimension $n+1$ condition.

\vskip0.5cm

\textbf{Acknowledgement.} This project strongly profited from many discussions among the authors during the HIM summer school ``Unique continuation and inverse problems" and the MPI MIS summer school ``Inverse and Spectral Problems for (Non)-Local Operators" at which the authors participated. The authors would like to thank that Hausdorff Center for Mathematics and the Max-Planck Institute for Mathematics in the Sciences for their support during these two weeks. Y.-H. Lin is supported by the Academy of Finland, under the project number 309963.

\bibliographystyle{alpha}
\bibliography{ref}

\end{document}